\documentclass{amsart} \usepackage{amssymb,amsfonts,amscd}
\newtheorem{theorem}{Theorem}[section]
\newtheorem{lemma}[theorem]{Lemma}
\newtheorem{corollary}[theorem]{Corollary}
\newtheorem{proposition}[theorem]{Proposition}
\theoremstyle{remark}

\theoremstyle{definition}
\newtheorem{definition}[theorem]{Definition}
\newtheorem{example}[theorem]{Example}
\numberwithin{equation}{section} \makeatother

\DeclareMathOperator{\Kdb}{{\mathbb K}}
\DeclareMathOperator{\A}{{\it A}}
\DeclareMathOperator{\B}{{\it B}}
\DeclareMathOperator{\Cdb}{{\mathbb C}}
\DeclareMathOperator{\Rdb}{{\mathbb R}}
\DeclareMathOperator{\Ddb}{{\mathbb D}}

\DeclareMathOperator{\Ndb}{{\mathbb N}}

\begin{document}
\title[Ideals and structure of operator algebras]{Ideals
and structure of operator algebras}

\author[M. Almus, D. P. Blecher,  {\protect \and}
S. Sharma]{Melahat Almus, David P. Blecher,  {\protect \and}Sonia Sharma}

\address{Department of Mathematics, University of Houston, Houston, Texas, 
TX 77204-3008, USA} \email[Melahat Almus]{almus@math.uh.edu} \email[David P.
Blecher]{dblecher@math.uh.edu} \email[Sonia
Sharma]{sonia@math.uh.edu} 
\thanks{Partially supported by a grant DMS 0800674 from
the National Science Foundation.}
\begin{abstract}
We continue the study of r-ideals, $\ell$-ideals, and HSA's in
operator algebras.  Some applications are made to the structure of
operator algebras, including Wedderburn type theorems for a class of
operator algebras.  We also consider the one-sided $M$-ideal
structure of certain tensor products of operator algebras.
\end{abstract}
\maketitle

\section{INTRODUCTION}

An {\em operator algebra} (resp.\ {\em operator space}) is a norm
closed algebra (resp.\ vector space) of operators on a Hilbert space. 
The present paper is a continuation of a program (see e.g.\
\cite{Baus,BEZ,BHN,BZ,Hay,Sharma})
 studying the
structure of operator algebras and operator spaces using `one-sided
ideals'. We shall have nothing to say about general one-sided ideals
in an operator algebra $A$, indeed not much is known about general
closed ideals in some of the simplest classical function algebras.
However there is a tractable and often interesting class of
one-sided ideals in $A$, which corresponds to a kind of
`noncommutative topology' for $A$, and to a noncommutative variant
of the theory of peak sets and peak interpolation for function
algebras (see \cite{Hay,BHN}).  These are the {\em $r$-ideals},
namely the right ideals of $A$
 possessing a left contractive approximate identity (cai).
These ideals are in bijective correspondence with the {\em
$\ell$-ideals} (left ideals with right cai), and with the {\em
hereditary subalgebras (HSA's)} of $A$, defined below. Much of our
paper is a further development of the general properties and
behaviors of these objects, with applications to the structure of
operator algebras.
Thus in Section 2 we record many new general facts about one-sided
ideals and HSA's, as well as some other preliminary results. Section
3 mainly concerns the existence of nontrivial r-ideals, and of
maximal r-ideals. There are some interesting connections here with
the
remaining open problems from \cite{Hay,BHN,BEZ}.

In Section 4 we apply some of these ideas to
 study a class of operator algebras which we have
not seen in the literature, which we call {\em $1$-matricial
algebras.}  One of our original motivations for introducing these
algebras, is that they might perhaps lead to some insight into
important open questions about operator algebraic amenability and
related notions, for example because their second duals are also
easy to work with.   This class of algebras is large enough to
display some interesting behaviors,
 e.g.\ some $1$-matricial algebras are right ideals in their
biduals, and others are not.   Amongst other things we prove
Wedderburn type structure theorems for operator algebras, using
$1$-matricial algebras as the building blocks.  A key ingredient, as
one might expect given the history of Wedderburn type
decompositions, is played by minimal $r$-ideals.  In Section 5 we
continue this theme by deriving some new characterizations of
$C^*$-algebras consisting of compact operators.

 In the final section
we turn to the ideal structure of the Haagerup tensor product of
operator algebras. We also discuss some connections with the study
of operator spaces which are one-sided $M$- or $L$-ideals in their
biduals (initiated in \cite{Sharma}). Our results provide natural
examples of operator spaces which are right but not left ideals (or
$M$-ideals) in their second dual. Their duals are left but not right
$L$-summands in their biduals.

Turning to notation, we reserve the letters $H$ and $K$ for  Hilbert spaces.
  We will use basic concepts from operator space
theory, and from the operator space approach to operator algebras,
which may be found e.g.\ in \cite{BLM,Pis}.

A {\em topologically simple} algebra has no closed ideals. A {\em
semiprime} algebra is one in which $J^2 = (0)$ implies $J = (0)$,
 for closed ideals $J$.
For subsets $J, I$ of an algebra $A$ the {\em left} (resp.\ {\em right})
{\em  annihilators} are
 $L(J) = \{ a \in A : a J = (0) \}$ and $R(I) = \{ a \in A : I a = (0) \}$.
We recall that a Banach algebra $A$ is a {\em right} (resp.\ {\em left})
{\em  annihilator algebra} if $R(I) \neq (0)$
 (resp.\ $L(J)  \neq (0)$) for any proper closed left ideal $I$
(resp.\ right ideal  $J$) of $A$.   An {\em annihilator algebra} is
both a left and  right  annihilator algebra.
In \cite{Kap}, Kaplansky studied a class of algebras which he called
{\em dual}, these satisfy $R(L(J)) = J$ and $L(R(I)) = I$ for $J, I$
as in the last line. It is known that a $C^*$-algebra is an annihilator algebra
if and only if it is dual in Kaplansky's sense, and these are
precisely the $c_0$-direct sums of `elementary $C^*$-algebras',
where the latter term refers to the space of compact operators on
some Hilbert space. This class of  algebras has very many diverse
characterizations, some of which may be found in Kaplansky's works
or \cite[Exercise 4.7.20]{Dix}, and which we shall use freely. For
example, these are also the $C^*$-algebras $A$ which are ideals in their
bidual \cite{HWW}, or equivalently,
 $A^{**} = M(A)$.   We shall usually not use the word `dual', to
avoid any confusion with Banach space duality, and instead will call
these {\em annihilator $C^*$-algebras}, or {\em $C^*$-algebras
consisting of compact operators}.  In Section 4, we will be more
interested in several properties weaker than being an annihilator
algebra: {\em modular annihilator, compact, dense socle}.  See
\cite[Chapter 8]{Pal} for the definitions, and a thorough discussion
of these properties and their connections to each other.

A normed algebra is {\em unital} if it has an identity of norm $1$;
any operator algebra $A$ has a (unique) operator algebra unitization
$A^1$.
 A {\em bai} (resp.\ {\em cai}) is a bounded (resp.\ contractive)
two-sided approximate identity.  the existence of a bai implies
of course that $A$ is {\em left} and {\em right essential}, by which we mean
that the left or right multiplication by an element in $A$ induces
a bicontinuous injection of $A$ in $B(A)$.
An operator algebra $A$ is Arens
regular, and $A$ has a right bai (resp.\ right cai) iff $A^{**}$ has
a right identity (resp.\ right identity of norm $1$)--see e.g.\
\cite[Proposition 5.1.8]{Pal}, \cite[Section 2.5]{BLM}.
An algebra is {\em approximately unital} if it has a cai.
For an
approximately unital operator algebra $A$, we denote the left,
right, and two-sided, multiplier algebras by $LM(A), RM(A), M(A)$,
usually viewed as subalgebras of $A^{**}$ (see \cite[Section
2.5]{BLM}).
   If $A$ is an operator algebra, or
operator space containing the identity operator, then we write
$\Delta(A)$ for the {\em diagonal} $A \cap A^*$, and write $A_{\rm
sa}$ for the selfadjoint part of $\Delta(A)$.

The one-sided approximate identity in any $r$-ideal (resp.\
$\ell$-ideal)
 $J$ in an operator algebra $A$, converges weak* to the {\em support
projection} $p$ of $J$ in $A^{**}$, and $J^{\perp \perp} = pA^{**}$
(resp.\ $J^{\perp \perp} = A^{**}p$).  Indeed, we recall from
\cite{Hay,BHN} that $r$-ideals are precisely those right ideals of
the form $pA^{**} \cap A$ for a projection $p$ in $A^{**}$ which is
{\em open}.  By the latter term we mean that there is a net in
$pA^{**} p \cap A$ converging weak* to $p$; there are several other
equivalent characterizations in \cite{BHN}. Similarly for
$\ell$-ideals.   A {\em hereditary subalgebra (HSA)} of $A$ is an
approximately unital subalgebra $D$ of $A$ such that $DAD \subset D$:
 these are also the subalgebras of the form $pA^{**}p \cap A$ for
an open projection $p$ in $A^{**}$.

For subspaces $V_k$ of a vector space, we write $\sum_k \, V_k$ for the
vector subspace of finite sums of elements in $V_k$, for all $k$. A
{\em projection} in an operator algebra $A$ is always an orthogonal
projection. A projection in $A$ is called {\em $*$-minimal} if it
dominates no nontrivial projection in $A$.  We use this nonstandard
notation to distinguish it from the next concept: a projection or
idempotent $e$ in $A$ is called {\em algebraically minimal} if $eAe
= \Cdb e$.  Clearly an algebraically minimal projection  is
$*$-minimal.  In certain algebras the converse is true too, but this
is not common. In a semiprime Banach algebra the minimal left ideals
are all closed, and all of the form $Ae$ for an algebraically
minimal idempotent. Indeed, if $e$ is an idempotent then $Ae$ is a
minimal left ideal iff $e$ is algebraically minimal \cite{Pal}.
Similarly for right ideals.  Idempotents $e,f$ are {\em mutually orthogonal}
if $ef = fe = 0$.   Any $A$-modules appearing in our paper
will be operator spaces too; hence for morphisms we use
$CB(X,Y)_A$, the {\em completely bounded} right $A$-module maps.

We will also need in a couple of places some concepts from the
theory of operator space multipliers and one-sided $M$-ideals, which
can be found e.g.\ in \cite[Chapter 4]{BLM}, \cite{BEZ},  or the
first few pages of \cite{BZ}.  We write ${\mathcal M}_{\ell}(X)$ and
${\mathcal M}_{r}(X)$ for the unital operator algebras of {\em left
and right  operator space multipliers} of an operator space $X$. See
{\rm \cite[Chapter 4]{BLM}} for the definition of these. The
diagonal of these two operator algebras are
 the $C^*$-algebras ${\mathcal A}_{\ell}(X)$
and ${\mathcal A}_{r}(X)$ of {\em left and right adjointable
multipliers}. The operator space centralizer algebra $Z(X)$ is
${\mathcal A}_{\ell}(X) \cap {\mathcal A}_{r}(X)$.  The projections
in these three $C^*$-algebras are called respectively {\em left,
right,} and {\em complete, $M$-projections} on $X$. A subspace $J$
of $X$ is called, respectively, a {\em right, left, or complete,
$M$-ideal} in $X$, if $J^{\perp \perp}$ is, respectively, the range
of a left, right, or complete, $M$-projection on $X^{**}$. The right
(resp.\ left, complete) $M$-ideals in an approximately unital
operator algebra are precisely the $r$-ideals (resp.\ $\ell$-ideals,
closed ideals with cai).

If $X$ is an operator space then $C_\infty(X)$ (resp.\ $R_\infty(X)$)
denotes
the operator space of (countably) infinite columns (resp.\ rows) with
entries in $X$, formed by taking the closure of those columns
(resp.\ rows) with only finitely many nonzero entries.
Also, if $E_k$ is a subspace of $X$ for each $k$, then the
`column sum' $\oplus^c_k \, E_k$ consists of the
tuples $(x_k) \in C_\infty(X)$ with $x_k \in E_k$ for each $k$.
Similarly for the `row sum' $\oplus^r_k \, E_k$.

\section{HSA'S AND $r$-IDEALS}

In this section we collect several new results about $r$-ideals and
HSA's. The first is a generalization of \cite[Proposition 6.4]{BEZ}.

\begin{proposition}  \label{lm}  Let $A$ be an operator algebra with
right cai $(e_t)$.  Then we have ${\mathcal M}_r(A) = CB_A(A) = RM(\{ a \in
A : e_t a \to a \})$.  Also, the
 left $M$-ideals of $A$ are the
$\ell$-ideals in $A$. The  left $M$-summands of $A$ are precisely
the left ideals of form $Ae$, for a projection $e$ in the multiplier
algebra of $\{ a \in A : e_t a \to a \}$; and also coincide with the
ranges of completely contractive idempotent left $A$-module maps on
$A$.
\end{proposition}

\begin{proof}  That  ${\mathcal M}_r(A) = CB_A(A) = RM(\{ a \in A : e_t a \to
a \})$  follows from \cite[Theorem 6.1]{Baus}.  We are using the
quite nontrivial fact that the hypothesis of \cite[Theorem 6.1
(5)]{Baus} is removable, which was established in \cite{BHN}. This, together with Proposition
5.1 in \cite{BHN}, proves the summand assertion. It follows that if
$A$ has a right identity of norm $1$, then the right $M$-projections
correspond bijectively to the projections in $A$. Using this, the
proof of the left $M$-ideal assertion is just as in
\cite[Proposition 6.4]{BEZ}.
\end{proof}

{\sc Remark.}  It is not hard to see that the last result is not
true for general operator algebras.  Indeed, right $M$-summands or
right $M$-ideals of operator algebras with a right cai, will be
right ideals, but they need not have a right cai.

\begin{lemma}  \label{aus}  An $r$-ideal which has a left identity
has a left identity of norm $1$.
\end{lemma}

\begin{proof}   Follows from \cite[Corollary 4.7]{Baus}.
\end{proof}

\begin{proposition}  \label{idl} If $A$ is an operator algebra with right cai,
and if $p$ is  a projection in $A^{**}$ such that $p A \subset A$,
then $A p \subset A$.
\end{proposition}

\begin{proof} As in the proof of Proposition 5.1 in \cite{BHN}.
\end{proof}

\begin{proposition} \label{spp}  Semiprimeness and topological simplicity
pass to (approximately  unital) HSA's of  operator algebras.
\end{proposition}

\begin{proof}  Suppose that  $D$ is an HSA in
an operator algebra $A$.  If $A$ is  semiprime and $J$ is an ideal
in $D$ with $J^2 = (0)$, then since $D$ is approximately  unital we
have $$J A J \subset J D A D J \subset J D J \subset J^2 = (0) .$$
Thus $A J A$ is a nil ideal in $A$, hence is zero, so that $J
\subset D J D \subset A J A = (0)$.

If $A$ is topologically simple and $J$ is a closed ideal in $D$,
then $A J A = (0)$ or $A J A = A$.  In the first case, $J = D J D =
(0)$. In the second case, $$D = D^3 \subset DAJAD \subset DADJDAD
\subset J .$$  So $D = J$, and so $D$ is topologically simple.
\end{proof}

{\sc Remark.}  We do not know if semisimplicity passes to HSA's.

\medskip

 We now state a series of simple results about the diagonal
of an operator algebra.

\begin{proposition}  \label{sdf}  For any operator algebra $A$, we have
  $\Delta(A) =
\Delta(A^{**}) \cap A$.
\end{proposition}

\begin{proof}  Suppose that $A$ is a subalgebra of a $C^*$-algebra $B$.
Clearly $\Delta(A)  \subset \Delta(A^{**}) \cap A$. If $x \in
\Delta(A^{**}) \cap A$, we may write $x = x_1 + i x_2$ with $x_k$
selfadjoint in $\Delta(A^{**})$.  Then $x + x^* = 2 x_1$, so that
$x_1 \in B \cap A^{\perp \perp} = A$.  Since $x_1$ is selfadjoint it
is in $\Delta(A)$.  Similarly for $x_2$, so that $x \in \Delta(A)$.
\end{proof}

\begin{corollary}  Let $A$ be an
operator algebra. If $A$ is an ideal in its bidual, then $\Delta(A)$
is an ideal in $\Delta(A^{**})$, and also in $\Delta(A)^{**}$.  Thus
$\Delta(A)$ is an annihilator $C^*$-algebra.
\end{corollary}

\begin{proof}  Since
$\Delta(A^{**}) A \Delta(A^{**}) \subset A^{**} A A^{**} \subset A$,
the first assertion follows.
The second assertion follows from the first, and the third from the
second.
\end{proof}

\begin{proposition} \label{prop} If $A$ is an HSA in
an approximately  unital operator algebra $B$, then $\Delta(A) =
\Delta(B) \cap A$, and this is a HSA in $\Delta(B)$ if it is
nonzero.
\end{proposition}

\begin{proof} If $x \in \Delta(B) \cap A$, write $x = x_1 + i x_2$ with $x_k$ selfadjoint.
If $A = eB^{**}e \cap B$ then $x = ex_1 e + i e x_2 e$.  Thus $x_k =
e x_k e \in eB^{**}e \cap B = A$, and so $x \in \Delta(A)$.  Hence
$\Delta(B) \cap A = \Delta(A)$.  Clearly $\Delta(A) \Delta(B)
\Delta(A) \subset \Delta(B) \cap (A B A) \subset \Delta(B) \cap A =
\Delta(A)$, and so $\Delta(A)$ is a HSA in $\Delta(B)$ if it is
nonzero.
\end{proof}

{\sc Remarks.} 1) \  Idempotents $e$ in a HSA $D$ of $A$ are
algebraically minimal in $A$ iff they are algebraically minimal in
$D$, since $eAe = e^2 A e^2 \subset e D e \subset \Cdb e$.

\smallskip

2)  \  If $A$ is an approximately unital operator algebra, and  $p$
is a projection in $M(A)$ then $\Delta(pAp) = p \Delta(A) p$.  This
follows from Proposition \ref{prop}, or directly.

\medskip

For any operator algebra $A$, the diagonal  $\Delta(A)$ acts
nondegenerately on $A$ iff $A$ has a positive cai, and iff
$1_{\Delta(A)^{\perp \perp}} = 1_{A^{**}}$. The latter is equivalent
to $1_{A^{**}} \in \Delta(A)^{\perp \perp}$.  To see these last
equivalences, note that a  positive cai in $A$ will converge weak*
to both $1_{\Delta(A)^{\perp \perp}}$ and $1_{A^{**}}$, so they are
equal.
 Conversely, if $1_{\Delta(A)^{\perp \perp}} = 1_{A^{**}}$,
and if $(e_t)$ is a cai for $\Delta(A)$,  then $e_t a \to a$ weak*,
hence  weakly, for all $a \in A$.  Thus $\Delta(A) A$ is weakly
dense in $A$, hence norm dense by Mazur's theorem.

\begin{proposition}  \label{dd}  If $A$ is an operator algebra
 such that $\Delta(A)$ acts nondegenerately on $A$,
then
$M(\Delta(A)) = \Delta(M(A))$.
\end{proposition}

\begin{proof}  For any approximately unital operator algebra
$A$, viewing multipliers of $A$ as elements of $A^{**}$, and
multipliers of $\Delta(A)$ in $\Delta(A)^{**} = \Delta(A)^{\perp
\perp} \subset A^{**}$, we have $\Delta(M(A)) \cap \Delta(A)^{\perp
\perp} \subset M(\Delta(A))$.   For if $T \in \Delta(M(A))$ and if
$a \in \Delta(A)$, then to see that $T a \in \Delta(A)$, we may
assume by linearity that $T$ and $a$ are selfadjoint. Then $T a \in
A$, but also $(T a)^* = a T \in A$, so that $T a \in \Delta(A)$.
Suppose that $A$ and $\Delta(A)$ share a common positive cai
$(e_t)$.  Then $T e_t \in \Delta(A)$, and so $T \in \Delta(A)^{\perp
\perp}$. Thus $\Delta(M(A)) \subset M(\Delta(A))$.

For  $T \in M(\Delta(A)), a \in A,$ we have $T a = \lim_t \, T e_t a
\in \Delta(A) A \subset A$.   Similarly $aT \in A$, and so the
$C^*$-algebra
 $M(\Delta(A))$ is a unital
subalgebra of $M(A)$. So $M(\Delta(A)) \subset \Delta(M(A))$.
\end{proof}

We define a notion based on the noncommutative topology of operator
algebras \cite{BHN,Hay}. We will say that an operator algebra $A$ is
{\em nc-discrete} if it satisfies the equivalent conditions in the
next result.

\begin{proposition}  \label{n95}  For an approximately unital operator algebra $A$
the following are equivalent:
\begin{itemize} \item [(i)] Every open projection $e$ in $A^{**}$
is also closed (in the sense that  $1-e$ is open).
\item [(ii)]  The open projections in $A^{**}$ are exactly the
projections in $M(A)$.
\item [(iii)]  Every  $r$-ideal $J$ of $A$ is of
the form $eA$ for a projection $e \in M(A)$.
\item [(iv)]  The left annihilator of every nontrivial $r$-ideal  of $A$ is a
 nontrivial $\ell$-ideal. \item [(v)]
Every  HSA of $A$ is of the form $eAe$ for a projection $e \in
M(A)$. \end{itemize} If any of these hold then $\Delta(A)$ is an
annihilator $C^*$-algebra.
\end{proposition}

\begin{proof}  If a projection $p$ in $A$ is both open and closed, then $p \in
M(B)$ for any $C^*$-cover  $B$ of $A$ by \cite[Theorem 3.12.9]{Ped}
and \cite[Theorem 2.4]{BHN}. Hence $p A \subset A^{\perp \perp} \cap
B = A$. Similarly $Ap \subset A$, so that $p \in M(A)$. So (i)
implies (ii), and the converse follows from the fact that any
projection in $M(A)$ is open \cite{BHN}.   The equivalence of (i)
and (ii)
 with (iii) and (v)
follows from basic correspondences from \cite{BHN,Hay}.  That  (iii)
implies (iv) follows from the fact that here $L(eA) = Ae^{\perp}$, an
$r$-ideal. Conversely, if (iv) holds, and $e \neq 1$ is an open
projection, then we claim that there is a nonzero open projection
$f$ with $f \leq e^\perp$.  Indeed choose such $f$ to be the support
projection of $L(J)$ where $J = A^{**} e \cap A$.  If $f = e^\perp$
then (i) follows, so assume that $e + f \neq 1$.  Since $e + f$ is
open, by the claim there exists a nonzero open projection $p$ with
$p \leq (e+f)^\perp$.  Then the HSA $D = p A^{**} p \cap A \subset
L(J) \cap R(L(J))$, so that $D = D^2 = (0)$. Thus  $p = 0$, a
contradiction.

If (i)--(v) hold, then every open projection in $\Delta(A)^{**}$ is
in $\Delta(M(A))$, hence in $M(\Delta(A))$ by the argument in
Proposition  \ref{dd}. Thus every left ideal $J$ of $\Delta(A)$ is
of the form $e \Delta(A)$ for a projection $e \in M(\Delta(A))$,
hence has a nonzero right annihilator. Thus $\Delta(A)$ is an
annihilator $C^*$-algebra.
\end{proof}

{\sc Remarks.}  1) 
\  Every finite dimensional unital operator algebra is obviously
nc-discrete.  We will see more examples later.

\smallskip

2) \  Of course the `other-handed' variants of (iii) and (iv) in the
Proposition are also equivalent to the others, by symmetry.  For a
projection $e \in M(A)$, the $\ell$-ideal and HSA corresponding to
the r-ideal $eA$, are $Ae$ and $eAe$ respectively, by the theory in
\cite{BHN}.

\bigskip

We will say that an operator algebra  $A$ is {\em $\Delta$-dual} if
$\Delta(A)$ is a dual $C^*$-algebra in the sense of Kaplansky, and
$\Delta(A)$ acts nondegenerately on $A$.

We briefly discuss the connections between the `$\Delta$-dual' and
the `nc-discrete' properties.  The flow is essentially one-way
between these properties. Being $\Delta$-dual certainly is far from
implying nc-discrete.  For example the disk algebra $A(\Ddb)$ has no
nontrivial projections and is $\Delta$-dual; but it has many
nontrivial approximately unital ideals (such as $\{ f \in A(\Ddb) :
f(1) = 0 \}$), thus is not nc-discrete.  However nc-discrete implies
$\Delta$-dual under reasonable conditions. For future use we record
some necessary and sufficient conditions for a  nc-discrete algebra
to be $\Delta$-dual.

\begin{corollary}  \label{n13c}  Let $A$ be an approximately unital
operator algebra which is nc-discrete. The following are equivalent:
\begin{itemize} \item [(i)]  $A$ is
$\Delta$-dual.
 \item [(ii)]  $\Delta(A)$
acts nondegenerately on $A$.
\item [(iii)] Every nonzero  projection in $M(A)$ dominates a nonzero
positive element in $A$.
\item [(iv)]  If $p$ is a nonzero  projection in $M(A)$,
then there exists $a \in A$ with $pap$ selfadjoint.
\item [(v)]  $1_{A^{**}} \in \Delta(A)^{\perp \perp}$.  
\end{itemize}
 \end{corollary}

\begin{proof}
We said in Proposition  \ref{n95} that $\Delta(A)$ is an annihilator
$C^*$-algebra.  Thus (i) $\Leftrightarrow$ (ii); and (ii)
$\Leftrightarrow$ (v) by the arguments above the statement of
Proposition
 \ref{dd}.
    Clearly
(iii) $\Rightarrow$ (iv).

(i) $\Rightarrow$ (iii) \ As in the last line $p \in M(\Delta(A)) =
\Delta(A)^{\perp \perp}$, hence dominates a projection in
$\Delta(A)$.

(iv) $\Rightarrow$ (ii) \  Since $\Delta(A)$ is a $c_0$-sum of
$C^*$-algebras of compact operators, every projection $p$ in
$\Delta(A)^{**} = M(\Delta(A))$ is open in $A^{**}$, hence lies in
$M(A)$ by Proposition \ref{n95} (ii).
In particular, $e_\Delta$, the identity of $\Delta(A)^{**}$ is in
$M(A)$. Hence $e_\Delta = 1$ (since $(1-e_\Delta)A (1-e_\Delta)$ cannot
contain any nonzero (selfadjoint) element in $\Delta(A)$). So $A$
has a positive cai.
\end{proof}

Nc-discrete algebras are reminiscent of Kaplansky's `dual algebras',
in that the `left (resp.\ right) annihilator' operation is a lattice
anti-isomorphism between the lattices of one-sided $M$-ideals of
$A$:

\begin{corollary} \label{elli2}   Let $A$ be a nc-discrete approximately unital operator algebra.
If $J$ is an r-ideal (resp.\ $\ell$-ideal) in $A$,  then the left
annihilator $L(J)$ (resp.\ right annihilator $R(J)$), equals $A e^\perp$ (resp.\
$e^\perp A$) where $e$ is the support projection of $J$.  This
annihilator is an $\ell$-ideal (resp.\ r-ideal), and $R(L(J)) = J$
(resp.\ $L(R(J)) = J$).   Also, the intersection, and the closure of
the sum, of any family of r-ideals (resp.\ $\ell$-ideals) in $A$ is
again a r-ideal (resp.\ $\ell$-ideal).
\end{corollary}

\begin{proof}  The first statements are evident.  The last statement is always true for the closure of the sum
\cite{BZ}.  
An intersection $\cap_i \, J_i$ of $r$-ideals $J_i$ is an $r$-ideal,
since it equals $R(\overline{\sum_i \, L(J_i)})$, and
$\overline{\sum_i \, L(J_i)}$ is an $\ell$-ideal.
\end{proof}

\begin{proposition} \label{elli}  If $A$ is an operator algebra,
which is an $\ell$-ideal in its bidual, then every projection $e \in
A^{**}$ is open, and the r-ideals (resp.\ $\ell$-ideals) in $A$ are
precisely the ideals $eA$ (resp.\ $Ae$) for projections $e \in
A^{**}$.

 If in addition $A$ is
approximately unital, then $A$ is nc-discrete, and every projection
in $A^{**}$ is in $M(A)$.
\end{proposition}

\begin{proof}   By Proposition \ref{idl}, any  projection in $A^{**}$ is
in the idealizer  of $A$ in $A^{**}$, and is open by an argument
similar to that on p.\ 336 of \cite{BHN}. Also,  $eA$ is an r-ideal,
and $Ae$ is an $\ell$-ideal, by Theorem 2.4 in \cite{BHN}. The last
statement is now easy.
\end{proof}

If in the first paragraph of the statement of the last result, $A$
also happens to be semiprime, then it is automatically approximately unital,
hence the second paragraph of the statement holds automatically.
This follows from:

\begin{proposition} \label{idbs}  A semiprime Arens regular Banach
algebra $A$  which is a left ideal in its bidual, has a bai (resp.\
cai) if it has a right  bai (resp.\  right cai). 
 \end{proposition}

\begin{proof}   Suppose that $(e_t)$ is a right bai for $A$, with
 $e_t \to e \in A^{**}$ weak*.  If we can  show that $e$ is an identity for
$A^{**}$, then $A$ has a bai. Let $J = \{  e y - y : y \in A \}
\subset A$. Then $A J = (0)$, so that $J^2 = (0)$, and $JA \subset
J$. Thus $J = (0)$, so that $e y = y$ for all $y \in A$. Hence
$A^{**}$ has an identity.  Similarly in the cai case.
\end{proof}

The next result generalizes part of Proposition  \ref{lm}.
 Although we will not really use it in the sequel, it is of independent
interest, extending a result of Lin on $C^*$-modules \cite{Lin}. See
\cite[Theorem 2.3]{BK} for the `weak* version' of the result. We
refer to  \cite{Bghm} for most of the notation used in the following
statement and proof.

\begin{theorem} \label{rig}  If $Y$ is a right rigged module in the
sense of {\rm \cite{Bghm}} over an approximately unital operator
algebra $A$, then ${\mathcal M}_{\ell}(Y) = CB(Y)_A = LM(\Kdb(Y)_A)$
completely isometrically isomorphically.
\end{theorem}

\begin{proof}  By facts from the theory of multipliers of an operator space
(see e.g.\ \cite[Chapter 4]{BLM}), the `identity map'
is a  completely contractive homomorphism ${\mathcal M}_{\ell}(Y)
\to CB(Y)$, which maps into $CB(Y)_A$. Since by \cite[p.\ 34]{BMP},
$CB(Y)_A$ is an
operator algebra, and $Y$ is a left operator $CB(Y)_A$-module (with
the canonical action),  then by the aforementioned theory
there exist a completely contractive homomorphism $\pi : CB(Y)_A \to
{\mathcal M}_{\ell}(Y)$ with $\pi(T)(y) = T(y)$ for all $y \in Y, T
\in CB(Y)_A$.  That is, $\pi(T) = T$.  Thus $CB(Y)_A = {\mathcal
M}_{\ell}(Y)$.  Finally, $\Kdb(Y)_A$ is an essential left ideal in
$CB(Y)_A$: it is easy to see that the left regular representation of
$CB(Y)_A$ on $\Kdb(Y)_A$ is completely isometric.  Thus $CB(Y)_A
\subset LM(\Kdb(Y)_A)$ completely isometrically.  However $Y$ is
a left operator module over $\Kdb(Y)_A$, hence also over
$LM(\Kdb(Y)_A)$ (see Theorem 3.6 (5) in
\cite{Bghm} and 3.1.11 in \cite{BLM}), and so every $T \in LM(\Kdb(Y)_A)$
corresponds to a map in $CB(Y)_A$. \end{proof}

The following follows either from Proposition \ref{lm}, or from
Theorem \ref{rig} (since all r-ideals $J$ in an approximately unital
operator algebra $A$ are right rigged modules over $A$).

\begin{corollary} \label{elli3}  If $J$ is an r-ideal in an
approximately  unital operator algebra $A$ then ${\mathcal
M}_{\ell}(J) = CB(J)_A$.  In particular, if $e$ is a projection in
$A$ then ${\mathcal M}_{\ell}(eA) = CB(eA)_A = eAe$.
\end{corollary}

\section{EXISTENCE OF $r$-IDEALS}

 In \cite[Lemma 6.8]{BHN}  it is shown that if $A$ is a unital operator
algebra, and $x \in {\rm Ball}(A)$, then $J = \overline{(1-x) A}$ is
an $r$-ideal.  These ideals, which have also been studied by G.
Willis in a Banach algebra context (see e.g.\ \cite{Wil}), for us
correspond to the noncommutative variant of {\em peak sets} from the
theory of function algebras (see \cite[Proposition 6.7]{BHN} for the
correspondence).  Here and in the following we write $1$ for the
identity in $A^1$ if $A$ is a nonunital algebra.  The following enlarges
 this class of examples:

\begin{proposition}  \label{exri} If $A$ is an approximately  unital operator
algebra, which is an ideal in an operator algebra $B$, then
$\overline{(1-x) A}$ is an $r$-ideal in $A$ for all $x \in {\rm
Ball}(B)$.
\end{proposition}

\begin{proof}  We may assume that $B$ is unital and closed.
Then  $\overline{(1-x) B}$ is an $r$-ideal in $B$.
  Also, $A$ is a two-sided $M$-ideal in $B$.  By Proposition 5.30
(ii) in \cite{BZ}, $A \cap \overline{(1-x) B}$ is an $r$-ideal in
$B$, so has a left cai.  Clearly $\overline{(1-x) A} \subset A \cap
\overline{(1-x) B}$.  Conversely, if $a \in A$ is in
$\overline{(1-x)B}$, then since the latter has left cai 
 we have that $a \in
\overline{(1-x) B A} \subset \overline{(1-x) A}$. Thus
$\overline{(1-x) A}$ is an $r$-ideal.
\end{proof}

{\sc Remark.} This is false if $A$  has only a right cai, e.g.\
$C_2$.   We do not know if it is true if $A$ has a left cai, or is a
left ideal in $B$, or both.  (Added in June 2010: we now know this.)   
Note that if $A$ has a left identity $e$ of norm $1$ then
$\overline{(e-x) A}$ is an r-ideal in $A$ for any $x \in {\rm
Ball}(A)$; in this case $\overline{(e-x) A e} = \overline{(e-xe) A
e}$ is an r-ideal in $Ae$, so has a left  cai which serves as a
left cai for $\overline{(e-x) A} = \overline{(e-xe) A}$.

\bigskip

We recall that a right ideal $J$ of a normed algebra $A$ is {\em
regular} if there exists $y \in A$ such that $(1-y)A \subset J$.  We
shall say that $J$ is {\em 1-regular} if this can be done with
$\Vert y \Vert \leq 1$ and $y \neq 1$.
We do not know if every r-ideal in a unital operator algebra is
1-regular, this is related to a major open question in
\cite{Hay,BHN} concerning the noncommutative variant of the notion
of peak sets from the theory of function algebras. That is, if $A$
is a unital operator algebra, then is every closed projection (i.e.\
the `perp' of an open projection) in $A^{**}$ a $p$-projection in
the sense of \cite{Hay}? We recall that $p$-projections are a
noncommutative generalization of the $p$-sets, and peak sets, from
the theory of function algebras. By  \cite[Proposition 6.7]{BHN} and other principles from
\cite{Hay,BHN}, this open question is equivalent to
whether every r-ideal in  a unital operator algebra $A$ is the closure of a  union of (nested,
if one wishes) right ideals of the form $\overline{(1-a)A}$ for
elements $a \in {\rm Ball}(A)$?  This is true if $A$ is a unital
function algebra (uniform algebra), by results of Glicksberg and
others \cite{Hay}. There are many reformulations of this question in
\cite{BHN}.

\begin{proposition}  \label{exri2} If $A$ is a  unital operator
algebra, then every r-ideal in $A$ whose support projection is the
complement of a $p$-projection, is 1-regular.  If every closed
projection in $A^{**}$ is a $p$-projection (that is,
 if the major open problem from \cite{Hay,BHN} mentioned above
has an affirmative
solution),  then every r-ideal in $A$ is 1-regular.
\end{proposition}

\begin{proof}   The first statement follows easily from 
\cite[Theorem 6.1]{BHN}.  The second follows from the first and the
correspondence between r-ideals and closed projections \cite{Hay}.
 \end{proof}

{\sc Remarks.}  1) \ If $A$ is a unital function algebra (uniform
algebra) then indeed every r-ideal is 1-regular, by the comments
above.

\smallskip

2)  \ The open question alluded to above, is also related to the question of
proximinality of r-ideals in a unital operator algebra $A$ (an open
question from our project with Effros and Zarikian \cite{BEZ}). If a
closed projection in $A^{**}$ satisfies an `interpolating' condition
spelled out above Corollary 6.5 in \cite{BHN}, then the associated
r-ideal is proximinal and 1-regular (by Proposition \ref{exri2}
and/or facts in \cite[Section 6]{BHN}). Indeed, an r-ideal in $A$ is
proximinal iff the associated closed projection satisfies the
`interpolating' condition above Corollary 6.5 in \cite{BHN}, but
with $x = 1$ there. This follows from the method of proof of
\cite[Proposition 6.6]{BHN}.

\bigskip

  We will
say that a one-sided ideal in an algebra $A$ is nontrivial, if it is
not $A$ or $(0)$.   From Proposition \ref{exri}, one would usually
expect there to be many nontrivial $r$-ideals in an approximately
unital operator algebra. For example, one can show that in every
function algebra there are
plenty of nontrivial $r$-ideals. 
However consider for example the semisimple commutative two
dimensional operator algebra $A = {\rm Span}(\{ I_2 , E_{11} +
E_{12} \})$ in $M_2$. This algebra has only two proper ideals, and
neither is an $r$-ideal. The following results give  criteria for
the existence of nontrivial $r$-ideals, and make it apparent that
the operator algebras which contain nontrivial r-ideals are far from
being a small or uninteresting class.
We have not seen item (2) of the next result in the literature, but
probably it is well known in some quarters.

\begin{theorem} \label{nr} If $A$ is
a Banach algebra,
 and $x \in {\rm Ball}(A)$, then \begin{itemize}
\item [{\rm (1)}]  $\overline{(1-x) A} = (0)$ iff $x$ is a left identity of norm $1$ for $A$.
\item [{\rm (2)}]  If $A$ is unital then $1-x$ is left invertible in
$A$ iff $1-x$ is right invertible in $A$.  If $A$ is nonunital then
$x$ is left quasi-invertible in $A$ iff $x$ is right
quasi-invertible in $A$.  These statements are also equivalent to
any one, and all, of the following statements: $\overline{(1-x) A} =
A, (1-x) A = A$, $\overline{A(1-x)} = A$, and $A(1-x) = A$.
\item [{\rm (3)}]  If $A$ is unital then the statements in {\rm (2)} are also
equivalent to
the same statements, but with $A$ replaced throughout by the Banach
subalgebra generated by $1$ and $x$.
\item [{\rm (4)}]  If $A$ has a left identity of norm $1$ write this element as $e$,
else set $e = 0$.  Every 1-regular right ideal of $A$ is trivial iff the
spectral radius $r(a) < \Vert a \Vert$ for all  $a \in A \setminus
\Cdb e$; and iff ${\rm Ball}(A) \setminus \{ e \}$ is composed
entirely of quasi-invertible elements.
\end{itemize}
\end{theorem}

\begin{proof}   (1) \  Clear.

(2) and (3) \ If $A$ is unital  then $\overline{(1-x) A}$ has a left
bai $(e_n)$, defined by $e_n = 1- \frac{1}{n} \sum_{k=1}^n \, x^k$.
Let $B$ be the closure of the algebra generated by $1$ and $x$. Then
$(e_n) \subset (1-x) B$.  Suppose that $\overline{(1-x) A} = A$.
Clearly $\overline{(1-x) B} \subset B$. Conversely, if $b \in B
\subset \overline{(1-x) A}$, then $e_n b \to b$, so that $b \in
\overline{(1-x) B}$. Thus $\overline{(1-x) B} = B =
\overline{B(1-x)}$.  Therefore $1-x$ is invertible in $B$, and hence also
in $A$, by the Neumann series lemma.  Conversely, if $1-x$ is right
invertible in $A$ then $(1-x) A = A$.  This proves the unital case
of (2) and (3).

For a nonunital Banach algebra $A$, let $A^1$ be a unitization of
$A$.  If $\overline{(1-x) A} = A$ then $x \in \overline{(1-x)A}
\subset \overline{(1-x)A^1}$. If $(e_n)$ is a left bai for
$\overline{(1-x)A^1}$ as in the last paragraph, then $e_n (1-x) \to
1-x$, so $e_n = e_n (1-x) + e_n x \to 1-x + x = 1$.  Thus $A^1 =
\overline{(1-x)A^1}$, and so $x$ is quasi-invertible. Conversely, it
is clear that $x$ right quasi-invertible in $A$ implies that $(1-x)
A = A$.

(4) \
If the condition on the spectral radius holds, then for $a \in {\rm
Ball}(A) \setminus \Cdb e$, we have $r(a) < 1$, so that $a$ is
quasi-invertible, or equivalently $\overline{(1-a) A} = A$.  If $a
\in \Cdb e$ the corresponding ideals are clearly trivial.

Supposing every 1-regular right ideal is trivial, then $1 \notin
{\rm Sp}_A(x)$ for all $x \in A \setminus \Cdb e$ of norm $1$.
Multiplying $x$ by  unimodular scalars shows that the unit circle
does not intersect ${\rm Sp}_A(x)$.  Thus $r(x) < 1$. By scaling,
$r(x) < \Vert x \Vert$ for all  $x \in A \setminus \Cdb e$.

The rest of (4) is evident from the above.
\end{proof}

{\sc Remark.}  Of course the $r(x) < \Vert x \Vert$  condition in (4)
is equivalent to
 $a^n \to 0$
for all $a \in {\rm Ball}(A) \setminus \Cdb e$.

\medskip

For operator algebras one can refine the last result further.
The equivalences of (i)--(iii) below is probably a well known
observation.

\begin{proposition}   \label{nr2} Let  $A$ be a closed
 subalgebra of $B(H)$, and $x \in {\rm Ball}(A)$.

 If $A$ is unital, then  the following
 are equivalent (and are equivalent to the other conditions in
  {\rm (2)} and  {\rm (3)}  in the last theorem);
\begin{itemize}
\item [{\rm (i)}]   $1-x$ is invertible in $A$.
\item [{\rm (ii)}] $\Vert 1 + x \Vert < 2$.
\end{itemize}
If $A$ is nonunital, then $x$ is quasi-invertible in $A$ iff {\rm
(ii)} holds with respect to $A^{1}$.

If $A$ is any operator algebra then the conditions in {\rm (4)} of
the last theorem hold iff $\Vert 1 + x \Vert < 2$ for all $x \in
{\rm Ball}(A)$ with $x \neq e$,
 and iff $\nu(x) < \Vert x
\Vert$ for all $x \in A \setminus \Cdb e$.  Here $\nu$ is the
numerical radius in $A^1$, and $e$ is as in {\rm (4)} of
the last theorem.
\end{proposition}

\begin{proof}  (i)
$\Rightarrow$ (ii) \ Suppose that $\Vert I_H + x \Vert = 2$.
If $\zeta_n \in {\rm Ball}(H)$ satisfies $\Vert \zeta_n  + x \zeta_n
\Vert^2 \to 4$, then $\Vert \zeta_n \Vert$, $\Vert x \zeta_n \Vert$,
and ${\rm Re}(\langle x \zeta_n , \zeta_n \rangle)$, all converge
to $1$, and so
$\Vert \zeta_n  - x \zeta_n \Vert^2 \to 0$.  Thus $\Vert \zeta_n \Vert
\leq \Vert (1-x)^{-1} \Vert  \Vert  \zeta_n  - x \zeta_n \Vert \to 0$,
which is a contradiction.

(ii) $\Rightarrow$ (i) \ Let $b = \frac{x+1}{2}$.  If $\Vert x + 1
\Vert < 2$, then by the Neumann series lemma, $1-b$ is invertible in
any operator subalgebra containing $1$ and $x$, and in particular in
$A$.  Hence (i) holds.

If $\Vert x + 1 \Vert < 2$ for all $x \in {\rm Ball}(A) \setminus
\Cdb e$, then the unit circle does not intersect the
numerical range in $A^1$ of such $x$, since $|1 + \varphi(e^{i
\theta} x)| < 2$ for all states $\varphi$ on $A$.  So $\nu(x) < 1$.
By scaling,  $\nu(x) < \Vert x \Vert$ for all $x \in A \setminus
\Cdb e$.

The rest of the assertions of the theorem are obvious.
\end{proof}

\begin{proposition} \label{twi}   Let  $A$ be a nc-discrete approximately  unital
operator algebra.  If $J$ is a right ideal in $A$ then $J$ is a
regular $r$-ideal  in $A$ iff $J$ is 1-regular and iff $J = e^\perp
A$ for a projection $e \in A$. Also, the following are equivalent:
\begin{itemize} \item [(i)] Every 1-regular right ideal is trivial.
\item [(ii)]  $\Delta(A) \subset \Cdb 1$.
\item [(iii)]  $A$ contains no nontrivial  projections.
\end{itemize}
Also, $A$ has no nontrivial $r$-ideals
iff $M(A)$ contains no nontrivial projections.
   \end{proposition}

\begin{proof}  If $(1-y)A \subset J$  as above, and
 $J = eA$ for a projection $e \in M(A)$, then
$e (1-y) = 1-y$.  That is $e^\perp = e^\perp y \in A$. Conversely,
$e^\perp  \in A$ implies that $eA$ is 1-regular.

That (ii)  $\Rightarrow$ (iii), and (iii) $\Rightarrow$ (i), is now
clear.  To see that  (i) $\Rightarrow$ (ii), note that if
$\Delta(A)$ contains a hermitian $h$ then $\nu(h) = \Vert h \Vert$,
so $h \in \Cdb 1$ by the last assertion of
Proposition \ref{nr2}. Thus $\Delta(A) \subset \Cdb 1$.

The last part follows from Proposition \ref{n95} (iii).
\end{proof}

\begin{proposition} \label{raets}
Let $A$ be an 
operator algebra which contains nontrivial 1-regular ideals (or
equivalently, ${\rm Ball}(A) \setminus \{ 1 \}$  is not composed
entirely of
 quasi-invertible elements).  Then 
proper  maximal $r$-ideals of $A$ exist.  Indeed, if $y \in {\rm
Ball}(A)$ is not quasi-invertible then $(1-y)A$ is contained in a
proper ($1$-regular) maximal $r$-ideal. 
The unit ball of the intersection of the 1-regular maximal
$r$-ideals of $A$ is composed entirely of quasi-invertible elements
of $A$.
\end{proposition}
\begin{proof}  We adapt the classical route. For a non-quasi-invertible $y \in {\rm
Ball}(A)$, let $(J_t)$ be an increasing set of proper r-ideals,
 each containing $(1-y)A$.  Then  $J = \cup_t \, J_t$ is a right ideal
which does not contain $y$, or else there is a $t$ with  $a = ya +
(1-y)a \in J_t$ for all $a \in A$.  The closure $\bar{J}$ of $J$ is
an $r$-ideal since it equals the closure of $\cup_t \, \bar{J}_t$.
Also $\bar{J}$ is proper, since the closure of a proper regular
ideal is proper \cite{Pal}.  Thus by Zorn's lemma, 
$(1-y)A$ is contained in a  (regular) maximal $r$-ideal.  Let $I$ be
the intersection of the proper 1-regular maximal $r$-ideals.
If $y \in {\rm Ball}(I)$,
but $y$ is not quasi-invertible, then $y \notin \overline{(1-y)A}$.
Let $K$ be a maximal proper $r$-ideal containing
$\overline{(1-y)A}$.
Then $y \notin K$ (for if $y \in K$ then $a = (1-y) a + y a \in K$
for all $a \in A$), and $K$ is regular.
So $y \in I \subset K$, a contradiction. Hence every element of
${\rm Ball}(I)$, is  quasi-invertible.
  \end{proof}

{\sc Remarks.} 1) \ One may replace $r$-ideals in the last result
with $\ell$-ideals, or HSA's.

\smallskip

2)  \ In connection with the last result we recall from algebra that
the Jacobson radical is the intersection of all maximal (regular)
one-sided ideals.

\section{MATRICIAL OPERATOR ALGEBRAS}

In this section, we shall only consider separable algebras.
This is a blanket assumption, and will be taken for granted hereafter.
 This is only for convenience, the general case is almost identical.

If $(q_k)_{k=1}^\infty$ are elements in a normed
algebra $A$ then we say that $\sum_k \, q_k = 1$ {\em right strictly}
 if $\sum_k \, q_k a = a$ for all $a \in A$, with the
 convergence in the sense of nets (indexed by finite subsets
 of $\Ndb$). Similarly
for {\em left strictly}, and {\em strict convergence} means both left and right
strict convergence. If $A$ is a Banach algebra and the $q_k$ are idempotents  with
$q_j q_k = 0$ for $j \neq k$, and if $\sum_k \, q_k = 1$ strictly, then
$(\sum_{k=1}^n \, q_k)$ is an approximate identity  for
$A$.  If $A$ is left or right essential
(that is, $A \subset B(A)$ bicontinuously via left or right multiplication)
then $(\sum_{k=1}^n \, q_k)$ is a bai for $A$. Indeed
 by a simple argument involving the principle of uniform boundedness, there is
a constant $K$ such that $\Vert \sum_{k \in J} \, q_k \Vert \leq K$
for all finite sets $J$. (We are not saying that convergent nets
are bounded, but that convergent series have bounded
partial sums, which follows because the sets $J$ are finite.)   Now suppose that in addition, $A$ is a closed subalgebra
of $B(H)$.  In this case we claim that, up to similarity,
we may assume that the $q_k$ are orthogonal projections
 and that $(\sum_{k=1}^n \, q_k)$ is a cai for $A$.   Indeed, since by the above the finite partial sums of $\sum_k \, q_k$ are uniformly
bounded, by basic similarity theory there exists an invertible operator
 $S$ with $S^{-1} q_k S$ a
projection for all $k$.  To see this, note that the set
$\{ 1 - 2 \sum_{k \in E} \, q_k \}$, where   $E$ is any finite
subset of $\Ndb$, is a bounded abelian group of invertible operators.
Hence by Lemma XV.6.1 of \cite{DS}, there exists
an invertible $S$ with $S^{-1} (1 - 2 \sum_{k \in E} \, q_k) S$ unitary for every $E$.
This forces $e_k = S^{-1} q_k S$ to be a
projection for all $k$.  Then  $B = S^{-1}
A S$ is a subalgebra of $B(H)$ with a cai
$(\sum_{k=1}^n \, e_k)$, for a  sequence of mutually
orthogonal projections
$e_k$ in $B$.

Because of the trick in the last paragraph, we will usually suppose in the remainder of this section that
the idempotents $q_k$ are projections.  This corresponds to the
`isometric case' of the theory.  The `isomorphic case' of our theory below sometimes follows
from the `isometric case' by the similarity trick above, however  
we will make rarely mention this variant of the theory.

\begin{proposition} \label{n0}  Let $(q_k)_{k=1}^\infty$ be
mutually  orthogonal
projections in an operator algebra $A$.
Then $\sum_k \, q_k = 1$ right strictly iff $A$ is the closure of
$\sum_k \, q_k A$.  In this case $(\sum_{k=1}^n \, q_k)$ is a left
cai; and it is a cai iff $A$ is approximately unital, and iff $A$ is
also the closure of $\sum_k \, A q_k$.   \end{proposition}

\begin{proof}  The first part is obvious.  If $A$ is
approximately unital then $(\sum_{k=1}^n \, q_k)$ must converge
weak* to an identity $1$ for $A^{**}$. The closure of $\sum_k \, A q_k$ is an $\ell$-ideal
with support projection $e \geq \sum_{k=1}^n \, q_k$, so $e = 1$ and
$A$ is the closure of $\sum_k \, A q_k$.
\end{proof}

\begin{proposition} \label{n1}
 If $A$ is an Arens regular Banach algebra 
 with idempotents 
$(q_k)_{k=1}^\infty$ with $\sum_k \, A q_k$ or $\sum_k \, q_k A$ dense in $A$ (for
example, if  $\sum_k \, q_k = 1$ left or right strictly), then $A$
is a right ideal in $A^{**}$ iff $q_k A$ is reflexive for all $k$. If
$A$ is topologically simple and Arens regular, then $A$
is a right ideal in $A^{**}$ if  $e A$ is reflexive for some idempotent $e \in A$. 
\end{proposition}

\begin{proof}  This is probably well known, and so we will just sketch the proof.
We write `w-' for the weak topology.
  We use the fact that $T : X \to Y$ is w-compact iff
the w-closure of $T({\rm Ball}(X))$ is w-compact and iff
$T^{**}(X^{**}) \subset Y$ \cite{Meg}.  Suppose that
 $A$ is a right ideal in $A^{**}$.  We have ${\rm Ball}(q_kA) \subset q_k{\rm Ball}(A)$, and the
w-closure of the latter set is w-compact in $A$.  A little argument shows that it
is w-compact in $q_k A$. Since ${\rm Ball}(q_kA)$ w-closed in $q_k
A$, it is w-compact, and so $q_k A$ is reflexive. Conversely, if $q_k A$ is reflexive
then $L_{q_k} : A
\to A$ is w-compact since the w-closure  in $A$ of $L_{q_k}({\rm
Ball}(A))$ is contained in a multiple of ${\rm Ball}(q_kA)$, which is
w-compact in $q_k A$, hence in $A$.  Inside $LM(A)$, $A$ is the
closure of finite sums of term of the form $a L_{q_k}$ or $L_{q_k}
a$, for $a \in A$, which are all weakly compact.  Since
the weakly compact operators form an ideal, left multiplication by
any element of $A$ is weakly compact, and hence
$A$ is a right ideal in $A^{**}$.   We only need $k =
1$ in the above argument if $A$ is topologically simple, for then $A q_1 A$ is dense in
$A$.
\end{proof}

\begin{definition}  \label{dfm}   We now define a class of examples which fit in the above
context.  We say that an operator algebra $A$ is {\em matricial}
if it has a set of matrix units $\{ T_{ij} \}$, whose
span is dense in $A$.  Thus $T_{ij} T_{kl} = \delta_{jk} T_{il}$, where
$\delta_{jk}$ is the Kronecker delta.  Define $q_k = T_{kk}$.
 We say that a matricial operator algebra $A$ is {\em $1$-matricial}
if $\Vert q_k \Vert = 1$ for all $k$, that is, iff
the $q_k$  are orthogonal projections.    We mostly focus
on $1$-matricial algebras
(other matricial operator algebras appear only
occasionally, for example in Corollary  \ref{n12}).  We
will think of two $1$-matricial algebras as being the same
if they are completely isometrically isomorphic.

As noted at the start of the section, we are only interested in
separable (or finite dimensional) algebras, and in this case we
prefer the following equivalent description of $1$-matricial
algebras.  Consider a (finite or infinite) sequence $T_1, T_2,
\cdots$ of invertible operators on a Hilbert space $K$, with $T_1=
I$.  Set $H = \ell^2 \otimes^2 K = K^{(\infty)} = K \oplus^2 K
\oplus^2 \cdots$
 (in the finite sequence case, $H = K^{(n)}$).  Define $T_{ij} = E_{ij} \otimes T_i^{-1} T_j \in
B(H)$ for $i,j \in \Ndb$, and let $A$ be the closure of the span of
the $T_{ij}$.  Then $T_{ij} T_{kl} = \delta_{jk} T_{il}$, so
that these are matrix units for $A$.  Then $A$ is a
$1$-matricial algebra, and all separable or finite dimensional
$1$-matricial algebras are completely isometrically isomorphic
to one which arises in this way (by
the proof of Theorem \ref{chgk1} below).   Let $q_k = T_{kk}$, then $\sum_k \,
q_k = 1$ strictly.

A {\em $\sigma$-matricial} 
algebra
is a $c_0$-direct sum of 
$1$-matricial algebras. Since we only care about the separable case
these will all be countable (or finite) direct sums.  It would certainly be
better to call these $\sigma$-$1$-matricial algebras, or something
similar, but since we shall not really consider any other kind, we drop
the `$1$' for brevity.
\end{definition}

\begin{lemma}  \label{n4}   Any $1$-matricial algebra $A$ is approximately unital, topologically
simple,  hence semisimple and semiprime, and is a compact modular
annihilator algebra.  It is an HSA in its bidual, so has the unique
Hahn-Banach extension property in  {\rm \cite[Theorem 2.10]{BHN}}.
It also has dense socle, with the $q_k$ algebraically minimal
projections with $A = \oplus^c_k \, q_k A = \oplus^r_k \, A q_k$.
The canonical representation of $A$ on $A q_1$ is faithful and
irreducible, so that $A$ is a primitive Banach algebra. \end{lemma}

\begin{proof} Clearly $(\sum_{k=1}^n \, q_k)$ is a cai.  Also, $q_j A q_k =
\Cdb T_{jk}$. Thus $q_j A q_j = \Cdb q_j$, so the $q_k$ are
algebraically minimal  projections with $A = \oplus^c_k \, q_k A$ by
Theorem 7.2 of \cite{Bghm}.  Also $q_j A = \overline{{\rm Span}(\{
T_{jk} : k \in \Ndb \})}$. If $J$ is a closed  ideal in $A$, and $0
\neq x \in J$, then $q_j x \neq 0$, and $q_j x q_k \neq 0$, for some
$j, k$.  Hence $T_{jk} \in J$, and so $T_{pq} \in J$ for all $p,q
\in \Ndb$, since these are matrix units.  So $A = J$. Thus $A$ is
topologically  simple, hence semisimple and semiprime. Thus the $q_k A$
are minimal right ideals, so $A$ has dense socle, and by
\cite[Proposition 8.7.6]{Pal} we have that $A$ is a compact modular
annihilator algebra.  Note that $q_j A^{**} q_k \subset \Cdb T_{jk}
\subset A$, and so $A A^{**} q_k \subset A$ and $A A^{**} A \subset
A$.  Hence $A$ is a HSA in its bidual, so has the unique Hahn Banach
extension property in  {\rm \cite[Theorem 2.10]{BHN}}. The
representation of $A$ on $A q_1$ is faithful, since $aAq_1 = (0)$
implies $a Aq_1 A = (0)$, hence $aA = (0)$ and $a = 0$.  It is also
irreducible, since $A q_1$ is a minimal left ideal.
\end{proof}

{\sc Remark.}  If a Banach space $X$ has the unique Hahn Banach
extension property in   \cite[Theorem 2.10]{BHN}, then  by
\cite{Lim},
 it is {\em Hahn-Banach smooth} in $X^{**}$,
 hence it is a {\em  HB-subspace} of $X^{**}$, and $X^*$ has the
Radon-Nikodym property.
By the work of Godefroy and collaborators, if $X^*$ has the latter
property then
there is a unique contractive projection from $X^{(4)}$ onto
$X^{**}$, and $X^*$ is a strongly unique predual (see e.g.\
\cite{Sharma}).  Thus all of the above holds if $X$ is a
$\sigma$-matricial operator algebra.

\begin{corollary}  \label{n5} A $1$-matricial  algebra $A$ is a right (resp.\ left, two-sided)
ideal in its bidual iff $q_1 A$ (resp.\ $A q_1$, $q_1 A$ and $A
q_1$) is reflexive.  \end{corollary}

{\sc Remarks.}  1) \ It is known that semisimple (and many
semiprime) annihilator algebras are ideals in their bidual
\cite[Corollary 8.7.14]{Pal}.  In particular, a $1$-matricial
annihilator algebra is an ideal in its bidual.   We conjecture that
a $1$-matricial  algebra $A$ is bicontinuously isomorphic to
$\Kdb(\ell^2)$ iff it is an annihilator algebra. 

\smallskip

2) \  It is helpful to know that in any $1$-matricial  algebra,
 $(T_{1k}) = (E_{1k} \otimes T_k)$ is a monotone Schauder basis for $q_1
A$. 
Indeed, clearly the closure of the span of the $T_{1k}$
equals $q_1 A$, and if $n < m$ then
$$\Vert \sum_{k=1}^n \, \alpha_k T_{1k} \Vert^2
= \Vert \sum_{k=1}^n \, |\alpha_k|^2  T_{k} T_k^* \Vert \leq \Vert
\sum_{k=1}^m \, |\alpha_k|^2  T_{k} T_k^* \Vert = \Vert \sum_{k=1}^m
\, \alpha_k T_{1k} \Vert^2.$$

\medskip

The following is a first characterization of $1$-matricial  algebras.  Others
may be derived by adding to the characterizations of $c_0$-sums of
$1$-matricial  algebras below, the hypothesis of topological simplicity.

\begin{theorem} \label{chgk1}
If $A$ is a topologically simple left or right essential operator
algebra with a
sequence of nonzero algebraically
minimal idempotents  $(q_k)$ with $q_j q_k = 0$ for $j \neq k$, and
$\sum_k \, q_k = 1$ strictly, then $A$ is  similar to a
$1$-matricial  algebra.  If further the $q_k$ are projections, then $A$ is unitarily isomorphic to a
$1$-matricial  algebra.  \end{theorem}

\begin{proof} By the similarity trick at the start of this section,
we may assume that the $q_k$ are projections.
Since $A$ is semiprime,  the $q_k A$ are minimal right ideals,
and $A$ has dense socle.  Any nonzero $T \in B(q_j A,q_k A)_A$ is
invertible, and if $S, T \in B(q_j A,q_kA)_A$ then $T^{-1} S \in
\Cdb q_j$.  Thus $B(q_j A,q_kA)_A = \Cdb T$. We now show that the `left
multiplication map'  $\theta : q_j A q_k \to CB(q_k A,q_j A)_A$ is a
completely isometric isomorphism.  Certainly it is completely
contractive and one-to-one. Also, if $T \in CB(q_k A,q_j A)_A$ then
$T(q_k) \in q_j A q_k$ with $T(q_k) q_ka = T(q_ka)$.  Thus
$\theta^{-1}$ is a contraction, and similarly it is completely
contractive. Choose $0 \neq T_k \in q_1 A q_k$, with $q_1 = T_1$.
Write $T_{k}^{-1}$ for the inverse of $T_k \in B(q_j A,q_kA)_A$, an
element of $q_k A q_1$ with $T_{k}^{-1} T_{k} = q_k$ and $T_{k}
T_{k}^{-1} = q_1$. Then $T_{jk} = T_j^{-1} T_k \in q_j A q_k$, and
so $q_j A q_k = \Cdb T_j^{-1} T_k$. Any $a \in A$ may be
approximated first by $\sum_{k = 1}^n \, q_k a$, and then by
$\sum_{j,k = 1}^n \, q_k a q_j$. Thus $A$ is the closure of the span
of $\{ T_{ij} \}$, which are a set of matrix units for $A$.

  If $A \subset B(H_0)$ nondegenerately,
set $K_k = q_k(H_0) = T_k^{-1}(H_0)$, 
 and let $K = q_1 H_0$.   Then $K_k  \cong K$ via $T_k$ and
 $T_k^{-1}$.
 Since these are Hilbert spaces, there is a unitary $U_k : K_k \to
 K$, and hence a unitary isomorphism
$U : H_0 = \oplus^2_k \, K_k \to K^{(\infty)}$.   It is easy to see
$U A U^*$ is a $1$-matricial  algebra, so that $A$ is unitarily equivalent to a
$1$-matricial  algebra.
\end{proof}

\begin{lemma} \label{chgk2} Let $A$ be an infinite dimensional $1$-matricial  algebra.
Then  $A$ is completely isomorphic to $\Kdb(\ell^2)$ iff $A$ is
topologically isomorphic to a $C^*$-algebra (as Banach algebras),
and iff
 $(\Vert T_k \Vert \Vert T^{-1}_k \Vert)$ is bounded (where
$T_k$ is as in Definition {\rm \ref{dfm}}).
If $\Vert T_k \Vert \Vert T^{-1}_k \Vert \leq 1$ for all $k$
then $A = \Kdb(\ell^2)$ completely isometrically.
\end{lemma}

\begin{proof}  If $\rho : A \to \Kdb(\ell^2)$ is a bounded homomorphism
and $e_k = \rho(q_k)$,  then $e_k$ are finite rank idempotents. If
$\rho$ is surjective, then the idempotents $e_k$ are rank one, since they are
algebraically minimal.    They are mutually orthogonal and uniformly
bounded, so there is an invertible $S \in B(\ell^2)$ with $S^{-1}
e_k S $ rank one projections, say $S^{-1} e_k S = | \xi_k \rangle
\langle \xi_k |$, $\xi_k$ a unit vector in $\ell^2$.  Since these
projections are
mutually orthogonal, $(\xi_k)$ is orthonormal.  Then $\rho(T_k) =
e_1 \rho(T_k) e_k$, so that $S^{-1} \rho(T_k) S = \lambda_k | \xi_1
\rangle \langle \xi_k |$ for some scalars $\lambda_k$. There is a constant such that $\Vert T_k
\Vert \leq C |\lambda_k |$. Similarly, $S^{-1} \rho(T_k^{-1}) S =
\mu_k | \xi_k \rangle \langle \xi_1 |$, and $\Vert T_k^{-1} \Vert
\leq D |\mu_k|$.  Thus $\Vert T_k \Vert \Vert T^{-1}_k \Vert \leq C
D |\lambda_k \mu_k|$.   However since $q_1 = T_k T_k^{-1}$ we have
 $$S^{-1} \rho(q_1) S =
| \xi_1 \rangle \langle \xi_1 | = \lambda_k \mu_k | \xi_1 \rangle \langle \xi_1 |,$$
so that $\lambda_k \mu_k = 1$.  Hence $\Vert T_k \Vert \Vert T^{-1}_k \Vert
\leq C D$.

 Suppose that $\Vert T_k \Vert \Vert T^{-1}_k \Vert \leq M^2$
for each $k$.  By multiplying by appropriate constants, we may assume
that $\Vert T_k \Vert = \Vert T^{-1}_k \Vert \leq M$ for all $k$.
 Let $x = {\rm diag} \{T_1, T_2, \cdots \}$, and $x^{-1} = {\rm diag}
\{T_1^{-1}, T_2^{-1}, \cdots \}$.  The map $\theta(a) = x a x^{-1}$
is a complete isomorphism from $A$ onto $\Kdb(\ell^2)$, and
$\theta$ is isometric if $M = 1$.

If $A$ were isomorphic to a $C^*$-algebra $B$, then $B$ is a
topologically simple $C^*$-algebra with dense socle, so is a dual
algebra in the sense of Kaplansky. Since it is topologically
 simple, $B =
\Kdb(\ell^2)$.  It is a simple consequence of similarity theory that
a bicontinuous  isomorphism from $\Kdb(\ell^2)$ is a complete
isomorphism.
\end{proof}

An operator algebra will be called a {\em subcompact
$1$-matricial  algebra}, if it is
(completely isometrically isomorphic to) a
$1$-matricial algebra with
the space $K$ in the definition of a $1$-matricial algebra (the
second paragraph of \ref{dfm}) being finite dimensional.

 \begin{lemma} \label{n6}  A $1$-matricial algebra $A$ is subcompact
iff $A$ is completely isometrically                                              
        isomorphic to a subalgebra of $\Kdb(\ell^2)$, and 
iff its   $C^*$-envelope is an annihilator $C^*$-algebra.
In this case, $A$ is an ideal in its bidual, and
$q_k A$ (resp.\ $A
q_k$) is linearly completely isomorphic to a row (resp.\ column)
Hilbert space.  Here $q_k$ is as in
 Definition {\rm \ref{dfm}}.  Indeed, if a
$1$-matricial  algebra $A$ is bicontinuously
(resp.\ isometrically) isomorphic to
a subalgebra of $\Kdb(\ell^2)$, then $A$ is bicontinuously
(resp.\ isometrically)
isomorphic to a subcompact $1$-matricial algebra.
 \end{lemma}

\begin{proof} That a subcompact $1$-matricial algebra
is a subalgebra of $\Kdb(\ell^2)$ follows from the definition.
We leave it as an exercise that the separable operator algebras whose
 $C^*$-envelope is an annihilator $C^*$-algebra, are precisely
the subalgebras  of $\Kdb(\ell^2)$.  If $\theta$
was a bicontinuous homomorphism from $A$ onto a subalgebra of
$\Kdb(\ell^2)$, then $e_k = \theta(q_k)$ is a finite rank
idempotent. Hence $e_k \Kdb(\ell^2)$ is Hilbertian. Thus $q_k A$ is
Hilbertian, and similarly $A q_k$ is Hilbertian. These are
reflexive, and so  $A$ is an ideal in its bidual by Corollary
\ref{n5}. If $H_0$ is the closure of $\theta(A)(\ell^2)$, then the
compression of $\theta$ to $H_0$ is a nondegenerate bicontinuous
homomorphism, with range easily seen to be inside $\Kdb(H_0)$. So we
may assume that $\theta$ is nondegenerate from the start.
 As in Theorem \ref{chgk1}, there is an invertible operator $S$ on
 $\ell^2$ with $p_k = S^{-1} e_k S$ mutually orthogonal
projections.  These are compact, so finite dimensional.
Now appeal to Theorem \ref{chgk1} and its proof to see
that $A$  is bicontinuously isomorphic to a subcompact $1$-matricial  algebra.
The other cases are similar. 
   \end{proof}

{\sc Remarks.}  1) \ Suppose that in the last lemma, also $(\Vert T_k \Vert \Vert
T^{-1}_k \Vert)$ is unbounded, so that $A$ is not isomorphic to a
$C^*$-algebra as Banach algebras.  In this case $A$ is not amenable (nor
has the total reduction property of Gifford \cite{Gifford}, 
etc).  For amenability implies the total reduction property, and the
total reduction property for subalgebras of $\Kdb(\ell^2)$, implies
by \cite{Gifford} that $A$ is similar to a $C^*$-algebra. It is
interesting to ask if a $\sigma$-matricial algebra is amenable (or
has the total reduction property, etc) iff it is isomorphic to a
$C^*$-algebra.  Probably no $1$-matricial
algebras are amenable, biprojective, have Gifford's reduction
property, etc, unless it is isomorphic to a $C^*$-algebra,
but this needs to be checked.

\smallskip

2) \ We do not know if 
$1$-matricial algebras are bicontinuously isomorphic iff they are
completely isomorphic.

 \begin{example}  \label{n7}  Let $K = \ell^2_2$, and $T_k = {\rm diag}\{ k, 1/
k \}$. In this case by the above $A$ is an ideal in its bidual, but
is not topologically isomorphic to $\Kdb(\ell^2)$ as Banach
algebras. Here $q_1 A$ is a row Hilbert space and $A q_1$ is a
column Hilbert space. Note that $A$ is not an annihilator algebra by
\cite[Theorem 8.7.12]{Pal}, since $(q_1 A)^*$ is not isomorphic to
$A q_1$ via the canonical pairing.
\end{example}

 \begin{example}  \label{n8}  Let $K = \ell^2$, and $T_k = E_{kk} + \frac{1}{k} I$.
 Claim: $q_1 A$ is not reflexive.  Indeed the Schauder basis $(T_{1k})$
(see Remark 2 after Corollary  \ref{n5}) fails the first part of the
 well known
 two part test for  reflexivity \cite{Meg}, because $\sum_{k=1}^\infty \, T_{k} T_k^*$ converges
 weak* but not in norm.  Or one can see that $q_1 A \cong c_0$ by
  Lemma \ref{sb2} below.
 Here $T_k^{-1}$ has $k$ in all diagonal entries but one, which has
 a positive value $< k$.  It follows that
$A q_1$ is a column Hilbert space.  By Corollary  \ref{n5},  $A$ is a left ideal in its
bidual, but is not a right ideal in its bidual.  This is interesting
since any $C^*$-algebra which is a left ideal in its
bidual is also a right ideal in its bidual \cite{Sharma}.

Note that this is not an annihilator algebra by \cite[Theorem
8.7.12]{Pal}, since $(q_1 A)^*$ is not isomorphic to $A q_1$. Also $A$ is
not bicontinuously isomorphic to a subalgebra of $\Kdb(\ell^2)$
by Lemma \ref{n6}.
\end{example}

 \begin{example}  \label{soni}  
Let $K=\ell^2$ and $T_k=I-\sum _{i=1}^k
(1-\sqrt{\frac{i}{k}})E_{ii}$. It is easy to see that
$q_k A$ is a row Hilbert space
and $Aq_k$ is a column Hilbert space, for all $k$.
Thus $A$ is an ideal in its
bidual. Also $A$ is not isomorphic to $\Kdb(\ell^2)$ since $( \Vert
T_k \Vert \Vert T_k^{-1} \Vert )$ is unbounded. 
Some of the authors are currently using examples such as these to
test conjectures about $1$-matricial algebras.  For example,
an argument by the second and third author in \cite{Sharma2}
gives a negative answer to the question 
`if $q_k A$ is a row Hilbert space and $Aq_k$ is a column Hilbert
space for all $k$, then is $A$  subcompact?' 
Indeed, we show there
that the $C^*$-envelope of the present example is not an annihilator $C^*$-algebra
(so $A$ is not subcompact by Lemma \ref{n6}).
\end{example}

For a $1$-matricial  algebra $A$,
 if we only care about the norms on $q_1A, Aq_1$ (as opposed to $q_2
A$ etc.) then we also may assume that $T_k \geq 0$ for all $k$, by
replacing $T_k$ by $(T_k T_k^*)^{\frac{1}{2}}$. This does not change
the norm on $q_1A$ and $A q_1$. Note that if we are given {\em any}
not necessarily invertible $T_k \geq 0$, for all $k \in \Ndb$, then
we can set $S_k = g_k(T_k)$, where $g_k(t) =
\chi_{[0,\frac{1}{2^k})}(t) + t \chi_{[\frac{1}{2^k},\infty)}(t)$.
Then $S_k$ is a small peturbation of $T_k$ which is invertible.
Using this trick one can build $1$-matricial  algebras such that
$q_1 A$ is `very bad':

\begin{lemma} \label{sb2}    If $T_2, T_3, \cdots$ are arbitrary
elements of norm $\geq 1$ in an operator space $X$,
then there exists a $1$-matricial  algebra $A$ with $q_1 A$
bicontinuously isomorphic to $\overline{{\rm Span}} \{ E_{1k}
\otimes T_k  \} \subset R_\infty(X)$.  (One may suppose without
loss that $X$ contains
$I$, the identity operator on a Hilbert space on which
$X$ is represented, and set $T_1 = I.$) Also, $A$ may be chosen
such that $A q_1$ is a row Hilbert space, if $\sum_{k=2}^n \,
|\alpha_k|^4 \leq \Vert \sum_{k=2}^n \, |\alpha_k|^2  T_k^2 \Vert^2$
for all scalars $\alpha_k$ and $n \in \Ndb$.
\end{lemma}

\begin{proof}   Assume that $I \in X \subset B(K)$.
As above, we may assume $T_k \geq 0$, and form $S_k$ as described.
Also,  $T_k \leq S_k$, $\frac{1}{2^k} \, I \leq
S_k$, and $S_k - T_k \leq \frac{1}{2^k} I$ and $S^2_k - T^2_k \leq
\frac{1}{2^k} I$. Then $$\sum_{k=1}^n \, |\alpha_k|^2  T_k^2 \leq
\sum_{k=1}^n \, |\alpha_k|^2  S_k^2 \leq \sum_{k=1}^n \,
|\alpha_k|^2 T_k^2 + \sum_{k=1}^n \, \frac{|\alpha_k|^2}{2^k}  I
\leq \sum_{k=1}^n \, |\alpha_k|^2 T_k^2 + \sup_k  \,  |\alpha_k|^2
I,$$ so that
$$\Vert \sum_{k=1}^n \, |\alpha_k|^2  T_k^2 \Vert^{\frac{1}{2}}
\leq \Vert \sum_{k=1}^n \, |\alpha_k|^2  S_k^2 \Vert^{\frac{1}{2}}
\leq \Vert \sum_{k=1}^n \, |\alpha_k|^2 T_k^2 \Vert^{\frac{1}{2}} +
\sup_k |\alpha_k| .$$   If $1 \leq \Vert T_k \Vert$ then
the right hand side is
dominated by $2 \Vert \sum_{k=1}^n \, |\alpha_k|^2  T_k^2
\Vert^{\frac{1}{2}}$, so that 
$$\overline{{\rm Span}(\{E_{1k} \otimes S_k : k \in
\Ndb \})} \cong \overline{{\rm
Span}(\{E_{1k} \otimes T_k : k \in \Ndb \})}$$
bicontinuously  (hence they are reflexive or
nonreflexive simultaneously).   Thus we obtain a $1$-matricial  algebra $A$ formed from the
$S_k, S_k^{-1}$, with $q_1 A$
 bicontinuously isomorphic to the closure of the span of
$E_{1k} \otimes T_k$ in $R_\infty(X)$.

Assume that 
$\sqrt{\sum_{k=2}^n \, |\alpha_k|^4} \leq \Vert \sum_{k=2}^n \,
|\alpha_k|^2  T_k^2 \Vert \leq \Vert \sum_{k=2}^n \, |\alpha_k|^2
S_k^2 \Vert$.   Since $\Vert I + T \Vert = 1 + \Vert T \Vert$ if $T
\geq 0$, it is easy to see that we can replace all occurrences of
the symbols `$k=2$', in the last formula,  by `$k=1$'.  Let $R_k =
S_k \oplus t_k I$. Then
$$\Vert \sum_{k=1}^n \, |\alpha_k|^2  R_k^2 \Vert = \max \{
\Vert \sum_{k=1}^n \, |\alpha_k|^2  S_k^2 \Vert ,
\sum_{k=1}^n \, |\alpha_k|^2 t_k^2 \} .$$
The last sum is dominated by 
$\sqrt{\sum_{k=1}^n \, |\alpha_k|^4}$,
if $\sum_k \, t_k^4 \leq 1$, and so there is no change in the
norm on $q_1A$: $\Vert \sum_{k=1}^n \, |\alpha_k|^2  R_k^2 \Vert =
\Vert \sum_{k=1}^n \, |\alpha_k|^2  S_k^2 \Vert$.
 Moreover, $S_k^{-1} \leq 2^k I$, and so
$\Vert \sum_{k=1}^n \, |\alpha_k|^2  S_k^{-2} \Vert \leq
 \sum_{k=1}^n \, |\alpha_k|^2 (2^k)^2$.  Therefore
$\Vert \sum_{k=1}^n \, |\alpha_k|^2  R_k^{-2} \Vert =
\sum_{k=1}^n \, |\alpha_k|^2/t_k^2$ if $t_k \leq 1/2^k$.
A similar fact holds at the matrix level.
This forces $A q_1$ to be a row Hilbert space,
 since the map $(\alpha_1, \alpha_2 \cdots) \mapsto
[ \alpha_1 t_1 R_{1}^{-1} : \alpha_2 t_2 R_{2}^{-1} :
\cdots ]$ is a complete isometry from the finitely supported
elements in $\ell^2$, with its row operator space structure,
 into $A q_1$.  Thus for example
we may take $t_k = 1/2^k$, and obtain a $1$-matricial  algebra $A$ formed from the
$R_k, R_k^{-1}$, with $A q_1$ a row
Hilbert space, and $q_1 A$
 bicontinuously isomorphic to the closure of the span of
$E_{1k} \otimes T_k$ in $R_\infty(X)$.
\end{proof}

 \begin{example}  \label{oh}
An example of a $1$-matricial  algebra $A$ with $q_1 A$  reflexive (isomorphic
to $\ell^4$) but not isomorphic to a Hilbert space, and  $A q_1$  a
row Hilbert space, is obtained from Lemma \ref{sb2} by taking $T_2,
T_3, \cdots$ to be
 the canonical basis for $O\ell^2$.  Here $O\ell^2$ is Pisier's
operator Hilbert space \cite{Pis}, and $T_1 = I_H$ is the identity operator
 on $H$ where $O\ell^2 \subset B(H)$.    Here
$\Vert \sum_{k=2}^n \, |\alpha_k|^2  T_k^2 \Vert^{\frac{1}{2}} =
(\sum_{k=2}^n \, |\alpha_k|^4)^{\frac{1}{4}}$.

One may vary this example by replacing $O\ell^2$ with other
`classical' operator spaces, to obtain $1$-matricial algebras with
other interesting features.  \end{example}

{\sc Remarks.} 1) \ For a $C^*$-algebra $A$, it is well known that
every minimal left or right ideal is a Hilbert space (since it is a
$C^*$-module over $\Cdb e \cong \Cdb$), as is $A/J$ for a maximal
left or right ideal $J$ (any maximal left ideal is the left kernel
of a pure state $\varphi$, and then $A/J \subset A^{**} (1-p)$ where
$1-p$ is a minimal projection in $A^{**}$ (see p.\ 87 in
\cite{Ped}).  For a minimal projection $q$ in a $W^*$-algebra $M$,
$qMq$ is one dimensional since it is projectionless, so $Mq$ is a
Hilbert space as in the minimal  ideal argument above).

The above gives, in contrast,
 very nice (semisimple, etc)
approximately unital operator algebras $A$ with  an r-ideal
 $J$ (resp.\ $K$) which is maximal (resp.\  minimal)
amongst all the right ideals, such that $A/J$ (resp.\ $K$) is not Hilbertian, indeed
is not reflexive, or is reflexive but is not Hilbertian.   For if $A$ is
one of the $1$-matricial  algebras in our examples, one can show 
that $J = \overline{\sum_{k \neq 1} \, q_k A}$ is
 maximal amongst all the right ideals, and $A/J \cong q_1A$ and $A/(q_1A)
\cong J$.   The latter is because for example $L_{q_1} : A \to q_1A$ is
a complete quotient map with kernel $J$.

\smallskip

2) \ For a $1$-matricial  algebra $A$, one may consider the
associated Haagerup tensor product  $Aq_1 \otimes_h q_1A$, which is
also an operator algebra. Such algebras are considered in the
interesting paper \cite{Ari}.  Immediately several questions arise,
which may be important, such as if there are some useful sufficient
conditions for when this algebra is isomorphic to $A$ (if it is, then
$A$ is completely isomorphic to $\Kdb(\ell^2)$). 

\bigskip

\begin{proposition}  \label{n99}  Let $A$ be a $\sigma$-matricial algebra.
Then $A$ is a $\Delta$-dual algebra.
\end{proposition}

\begin{proof} 
 Clearly $A$ has a positive cai.  
 Also
$\Delta(\oplus^0_k \, A_k) = \oplus^0_k \, \Delta(A_k)$, so that we
may assume that $A$ is a $1$-matricial  algebra.  If $x \in
\Delta(A)$ then so is $q_i xq_j$ for any $i,j$.  Note that $q_i xq_j
\neq 0$ iff $T_i T_i^* \in \Cdb T_j T_j^*$.  If we assume, as we may
without loss of generality, that $\Vert T_k \Vert = 1$ for all $k$,
then the latter is equivalent to $T_i T_i^* = T_j T_j^*$.   This
gives an equivalence relation $\sim$ on $\Ndb$, and we may partition
into equivalence classes, $E_k$ say, each consisting of natural
numbers. Let $B_k$ be the closure of the span of the $E_{ij} \otimes
T_i^{-1} T_j$, for $i, j \in E_k$. These are $1$-matricial algebras,
which are selfadjoint, hence $C^*$-algebras.
 Thus $B_k \cong \Kdb(H_k)$ for a Hilbert space $H_k$.
Note that the relation $T_i T_i^* = T_j T_j^*$ above implies that
$\Vert T_i \Vert$ is constant on $E_k$, and by taking inverses we
also have $\Vert T_i^{-1} \Vert$ constant on $E_k$.
  Since $q_i xq_j = 0$ if $i$ and $j$ come from distinct equivalence
classes, $\Delta(A)$ decomposes as a $c_0$-sum $\Delta(A) =
\oplus^{0}_k \, B_k$.  Indeed, clearly $B_k \subset \Delta(A)$, and
any $x \in \Delta(A)_{\rm sa}$ is approximable by a selfadjoint
finitely supported matrix in $\Delta(A)_{\rm sa}$ , and hence by a
finite sum  of elements from the $B_k$. Hence $\Delta(A)$ is an
annihilator $C^*$-algebra.
\end{proof}

A pleasant feature of $1$-matricial  algebras, is that their second
duals have a simple form:

\begin{lemma} \label{sdm}  If $A$ is a $1$-matricial  algebra
defined by 
a system of  matrix units $\{ T_{ij} \}$ in $B(K^{(\infty)})$ as in
 Definition  {\rm \ref{dfm}}, then $$A^{**}
\cong \{ T \in B(K^{(\infty)}) : q_i T q_j \in \Cdb T_{ij} \; \;
\forall \, i, j \}.$$

Thus  $A^{**}$ is the collection of infinite matrices $[\beta_{ij}
T_i^{-1} T_j]$, for scalars $\beta_{ij}$, which are bounded
operators on $K^{(\infty)}$.  \end{lemma}

\begin{proof}  Write $N$ for the space on the right
of the last displayed equation.  This is weak* closed.  Suppose that
$A$ is represented nondegenerately on a Hilbert space $H$ in such a
way that $I_H = 1_{A^{**}} \in A^{**} \subset B(H)$, the latter as a
weak* closed subalgebra, with the $\sigma$-weak topology agreeing on
 $A^{**}$ with the weak* topology of $A^{**}$.
 Then we have $q_i A^{**} q_j =
\Cdb T_{ij}$.   That is,  $A^{**} \subset N$ completely
isometrically.  If $x \in N$, and if $x_n = (\sum_{k=1}^n \, q_k) x
(\sum_{k=1}^n \, q_k)$, then $x_n \in A$, and $x_n \to x$ WOT, hence
weak*.  Thus $x \in \overline{A}^{w*} = A^{**}$.
\end{proof}

\begin{lemma} \label{sdu} Let $A$ be a $\sigma$-matricial algebra.
  If $p$ is a projection in the second dual of $A$,
then $p$ lies in $M(A)$ and in $M(\Delta(A))$, and is thus open in
the sense of {\rm \cite{BHN}}. Hence $A$ is nc-discrete. Also,
$$\Delta(A^{**}) = \Delta(A)^{**} = M(\Delta(A)) = \Delta(M(A)) .$$
\end{lemma}

\begin{proof}  We may assume that $A$ is a $1$-matricial algebra.
Let $x \in \Delta(A^{**})_{\rm sa}$.  If $x_n = (\sum_{k=1}^n \,
q_k) x (\sum_{k=1}^n \, q_k)$, then $x_n \in \Delta(A)$, and $x_n
\to x$ WOT, hence weak*. Thus $x \in \Delta(A)^{\perp \perp}$. Hence
$$\Delta(A^{**}) \subset \Delta(A)^{**} = \Delta(M(A)) = M(\Delta(A))
,$$ using Proposition  \ref{dd}.  Therefore all of these sets are equal
since $\Delta(A)$ and hence $\Delta(A)^{\perp \perp}$, are subsets
of $\Delta(A^{**})$. Thus any projection $p \in A^{**}$ is in
$M(A)$, and hence is open by 2.1 in \cite{BHN}.
\end{proof}

{\sc Remark.}  By the above, and using also  the notation in the
proof of Proposition  \ref{n99}, for any $1$-matricial algebra $A$
we have $\Delta(A) = \oplus^0_k \, B_k$, where $B_k$ are
$C^*$-subalgebras of $\Delta(A)$ corresponding to the equivalence
relation $\sim$ on $\Ndb$, and $B_k \cong \Kdb(H_k)$ for a Hilbert
space $H_k$.  It follows by Lemma \ref{sdu} that $$\Delta(A^{**}) =
\Delta(A)^{**} = \oplus^\infty_k \, B_k^{**} \cong \oplus^\infty_k
\, B(H_k) .$$  Recall, by Lemma \ref{sdm}, we may write any element
of $A^{**}$ as a matrix $[\beta_{ij} T_i^{-1} T_j]$, for scalars
$\beta_{ij}$.  One may ask what this matrix looks like if $x \in
\Delta(A^{**})$.  In this case, $\beta_{ij} = 0$ if $i$ and $j$ are
in different equivalence classes for the relation $\sim$ discussed
in the proof of Proposition \ref{n99}. Indeed if $x = x^*$ then it
is easy to see
that
$$\beta_{ij} \, T_j T^{*}_j = \overline{\beta_{ji}} \,  T_i
T_i^* , \qquad \forall i, j .$$  Assume, as we may, that $\Vert T_k
\Vert = 1$ for all $k$.  Taking norms we see that $|\beta_{ij}| =
|\beta_{ji}|$.    It follows that $\beta_{ij} =
\overline{\beta_{ji}}$; and also, if $\beta_{ij} \neq 0$ then
 $i \sim j$.   Thus 
$\beta_{ij} = 0$ if $i$ and $j$ are in different equivalence
classes.

\begin{proposition}  \label{n97}  Let $A$ be an operator algebra
such that for every nonzero  projection $p$ in $A$, $pq \neq 0$ for
some  algebraically minimal projection $q \in A$.   Then every
$*$-minimal projection in $A$ is algebraically minimal. This holds
in particular for $\sigma$-matricial algebras.
\end{proposition}

\begin{proof}    If $p$ is $*$-minimal, $p q \neq 0$ as above, then $\frac{1}{t} p
q p$ is an algebraically minimal projection for some $t > 0$, and thus equals $p$.  Hence
$p$ is algebraically minimal.
\end{proof}

We now give some `Wedderburn type' structure theorems. See e.g.\
\cite{Hele,Kat} for some other operator algebraic `Wedderburn type'
results in the literature.

\begin{theorem}  \label{n13}  Let $A$ be an approximately unital
semiprime operator algebra.  The following are equivalent:
\begin{itemize} \item [(i)]  $A$ is completely isometrically isomorphic
 to a $\sigma$-matricial algebra.
\item [(ii)]  $A$ is the closure of $\sum_k \, q_k A$ for mutually orthogonal
algebraically minimal  projections $q_k \in A$.
   \item [(iii)]  $A$ is the closure of the joint span
of the minimal right ideals which are also 
r-ideals (these are the $q A$, for algebraically minimal projections
$q \in A$).
\item [(iv)]  $A$ is $\Delta$-dual, and every $*$-minimal
projection in $A$ is algebraically minimal. 
\item [(v)]  $A$ is $\Delta$-dual, and every nonzero projection in
$A$ dominates a nonzero algebraically minimal projection in $A$.
\item [(vi)]  $A$ is nc-discrete, and every nonzero
projection in $M(A)$ dominates a nonzero algebraically minimal
 projection in $A$. 
\item [(vii)] $A$ is nc-discrete, and every nonzero HSA $D$ in $A$
containing no nonzero projections of $A$ except possibly an identity
for $D$, is one-dimensional.
 \end{itemize}  
 \end{theorem}

\begin{proof}  Clearly (i) implies (ii), and (ii) implies (iii), and
(v) implies (iv).
Also fairly obvious are that (vi)  implies (v) (using Corollary  \ref{n13c} (iii)) and
(vii).

(ii) $\Rightarrow$ (i) \  By Proposition \ref{n0} we have $\sum_k \,
q_k = 1$ strictly.  We partition the $(q_k)$ into equivalence
classes $I_j$, according to whether $q_k A \cong q_j A$ or not. Note
that $q_k A q_j = (0)$ if $j, k$ come from distinct classes, by the
idea in the proof of Theorem \ref{chgk1} above. If $j, k$ come from
the same class then $q_k A q_j$ is one dimensional, $q_k A q_j =
\Cdb T_{kj}$ say.
 Let $e_j = \sum_{k \in I_j} \, q_k$.
Then either $e_j T_{pq} = 0 = T_{pq} e_j$  (if $j$ is in a different class to $p,q$) or $e_j
T_{pq} = T_{pq} = T_{pq} e_j$ (if $j$ is in
 the same class as $p,q$). So $e_j$ is in the center
of $M(A)$.  Then $B = e_j A$ is an ideal in $A$, and for $b \in B,
\sum_{k \in I_j} \, q_k b = \sum_{k}  \, q_k b = b$, and similarly
$b \sum_{k \in I_j} \, q_k = b$.  As in earlier proofs $B$ is
generated by a set of matrix units $T_{ij}$ which it contains, and
hence is topologically simple.  By Theorem \ref{chgk1}, $B = e_j A$ is a $1$-matricial
algebra.  The map $A \to \oplus^0_{j \in J} \, e_j A$ is a
  completely isometric isomorphism, since any $a \in A$ is approximable in norm
by finite sums of term of the form $q_j a q_k$, each contained in
some $e_jA$.

(iii) $\Rightarrow$ (v) \
 Given (iii), the joint support of all the algebraically minimal
 projections is $1$ (e.g.\ as in the proof of (v)  $\Rightarrow$ (ii)
below).  Thus the
 closure of the sum of the $q \Delta(A)$ for all the algebraically minimal
 projections $q$, is $\Delta(A)$ (since the weak* closure in the second
 dual contains $1$).  So by \cite[Exercise 4.7.20 (ii)]{Dix},
$\Delta(A)$ is an annihilator $C^*$-algebra with
 support projection $1$, and    hence $A$ is  $\Delta$-dual.   Let $e$ be a
nonzero projection in $M(A)$. Then since the joint support of all
algebraically minimal projections is $1$,
 $e q \neq 0$ for an
algebraically minimal projection $q$. 
We have $(e q e)^2 = e q e q e = t e q e$ for some $t
> 0$, so that $\frac{1}{t} e q e$ is an algebraically minimal
projection dominated by $e$.

 (v)  $\Rightarrow$ (ii) \
Let $(e_k)$ be a maximal family of mutually orthogonal algebraically
minimal projections in $A$, and let $e = \sum_k \, e_k \in A^{**}$.
If $A$ is $\Delta$-dual, and  $1$ is 
the identity in $A^{**}$, then $1$ is also the identity of
$M(\Delta(A)) = \Delta(A)^{**}$, so that $1 - e \in M(\Delta(A))$. 
Hence if $e \neq 1$ then  $1-e$ dominates a nonzero $*$-minimal
projection in $\Delta(A)$, which in turn dominates a nonzero
algebraically minimal projection in $A$, 
contradicting maximality of $(e_k)$.  So $\sum_k \, e_k
= 1$.  The closure $L$ of $\sum_k \, e_k A$ is $A$.  This is because
 $L^{\perp \perp}$
is the weak* closure of sums of the $e_k A^{**}$ by e.g.\ A.3 in \cite{BZ},
which contains $1$ and hence equals $A^{**}$.  So $L = A \cap L^{\perp \perp}
= A$.

(iii) $\Rightarrow$ (vi) \  We have by the above that (iii) implies
(i), which implies by Lemma \ref{sdu} that $A$ is nc-discrete. The
proof of (iii) implies (v) also gives the other part of (vi).

 (iv) $\Rightarrow$ (v) \ Given a
projection $e \in M(A)$, we have $e \in \Delta(M(A)) =
M(\Delta(A))$ by Proposition  \ref{dd}.
 Since $\Delta(A)$ is an annihilator $C^*$-algebra,
$e$ majorizes a nonzero $*$-minimal projection (since this is
true for algebras of compact operators), which by (iv) is
algebraically minimal.

(vii) $\Rightarrow$ (iv) \ 
As we said in Proposition \ref{n95}, $\Delta(A)$ is an
annihilator $C^*$-algebra.  Given a nonzero projection
$p \in M(A) \setminus A$, then either the HSA $pAp$ is one-dimensional,
in which case $p$ dominates the identity of $pAp$,
or it is not one-dimensional, in which case $p$ dominates
a nonzero projection in $A$ by (vii).  Thus
$A$ is $\Delta$-dual by Corollary \ref{n13c} (iii). Now (iv) is
clear.

 That (iii) is equivalent to (i) also follows
from Theorem \ref{n133}.
\end{proof}

{\sc Remark.}
If $A$ is a one-sided ideal in $A^{**}$, then
$A$ is nc-discrete by Proposition \ref{elli}.  In this case, one may
remove the condition `$A$ is nc-discrete' in (vi)--(vii), and one may
replace `$A$ is $\Delta$-dual' by `$\Delta(A)$ acts nondegenerately
on $A$' in (iv) and (v).

\bigskip

The following is another characterization of $\sigma$-matricial
algebras.

\begin{theorem}  \label{n13cc}  Let $A$ be an approximately unital
semiprime operator algebra such that $\Delta(A)$ acts
nondegenerately on $A$. 
 Suppose also that every $*$-minimal projection $p \in A$ is also minimal among
 all idempotents (that is, there are no nontrivial idempotents in $pAp$). The following are equivalent:
\begin{itemize} \item [(i)]  $A$ is completely isometrically isomorphic
 to a $\sigma$-matricial algebra.
\item [(ii)]  $A$ is compact.
\item [(iii)]  $A$ is a modular annihilator algebra.
\item [(iv)]  The socle of $A$ is dense.
\item [(v)]  $A$ is semisimple and the spectrum of every  element in $A$ has
no nonzero limit point. 
\end{itemize}
 \end{theorem}

\begin{proof}   We first point out a variant of (iv) in the
last theorem: if  $A$ is a  $\Delta$-dual algebra
which is semiprime, and if
every nonzero $*$-minimal
projection $p$ is minimal among the idempotents in $A$,
and $pAp$ is finite dimensional (or equivalently, $p$ is in the socle,
or is `finite rank'),
then $A$ is completely isometrically isomorphic
 to a $\sigma$-matricial algebra.  To see this, note that
by
Proposition \ref{spp}, $pAp$ is semiprime, hence is
one-dimensional, or equivalently $p$ is algebraically minimal (for
if not, then  by Wedderburn's theorem $pAp$ contains nontrivial
idempotents, contradicting the hypothesis).  Thus Theorem
\ref{n13} (iv) holds.

If (ii), (iii), (iv), or  (v) hold,
then it is known that $pAp$ is finite dimensional for every
projection $p \in A$ (some of these follow from the ideas in 8.6.4
and 8.5.4 in \cite{Pal}). Suppose that $p$ is $*$-minimal. By the
last paragraph, we will be done
if we can show that any one of (ii)--(v) imply that
$\Delta(A)$ is an annihilator $C^*$-algebra. If $A$ is compact then so is
$\Delta(A)$, hence it is an annihilator $C^*$-algebra. By 8.7.6 in
\cite{Pal}, (iv) implies (ii) and (iii).
 If $A$ is a modular
annihilator algebra then the spectrum condition in (v) holds by
8.6.4 in \cite{Pal}. If the spectrum condition in (v) holds, then by
the spectral permanence theorem, if $x \in \Delta(A)_{\rm sa}$ and
$B$ is a $C^*$-algebra generated by $A$, then ${\rm
Sp}_{\Delta(A)}(x) \setminus \{ 0 \}  = {\rm Sp}_{A}(x) \setminus \{
0 \}  = {\rm Sp}_{B}(x) \setminus \{ 0 \}$, which has no nonzero
limit point.  So $\Delta(A)$ is an annihilator $C^*$-algebra by
\cite[4.7.20 (vii)]{Dix}.
\end{proof}

 \begin{example}  In the last theorems, most of
the hypotheses seem fairly sharp, as one may see by considering
examples such as the disk algebra, 
or the following example (or Example \ref{n10}). Let $B
= R D R^{-1}$, where $D$ is the diagonal copy of $c_0$ in
$B(\ell^2)$, and $R = I + \frac{1}{2} S$ where $S$ is the backwards
shift.  Indeed $R$ could be any invertible operator such that the
commutant of $R^* R$ contains no nontrivial projections in $D$. This
example has most of the properties in Theorem \ref{n13cc}: its
second dual is isometrically identifiable with
$R \overline{D}^{w*} R^{-1}$ in $B(\ell^2)$, which is unital, and so $B$ is
approximately unital; $B$ is semiprime and satisfies (ii)--(v) in
Theorem  \ref{n13cc}, since $D$ does. Moreover, $B$ has no
nontrivial projections. Indeed,
 if $q = R p R^{-1}$ is  a projection then $p$ is an idempotent in
$D$, hence is a projection.  That $q = q^*$ implies that $p$ is in
the commutant of $R^* R$, which forces $p = 0 = q$. On the other
hand, $B$ does not satisfy (i) of Theorem \ref{n13cc}, hence has no
positive cai.  Thus it is  not $\Delta$-dual although it is
nc-discrete, indeed it is an ideal in its bidual, and its diagonal
$C^*$-algebra is an annihilator $C^*$-algebra. Thus this example
illustrates the importance of the condition that $\Delta(A)$ acts
nondegenerately on $A$ in the last theorem. One may vary this
example by letting $A = B \oplus \Cdb$, where $B = R D R^{-1}$ as
 above.  This has exactly one nontrivial  projection.
Variants of this example are also useful to illustrate
hypotheses in others of our results, such as replacing
 $D$ by the diagonal copy of $\ell^\infty$ in $B(\ell^2)$.
\end{example}

\begin{corollary}  \label{n12}  Let $A$ be a semiprime
left or right essential operator algebra, containing  
algebraically minimal idempotents $(q_k)_{k=1}^\infty$ with $q_j q_k
= 0$ for $j \neq k$, and $\sum_k \, q_k = 1$ strictly. Then
$A$ is semisimple and $A$ is  completely isomorphic to a
$\sigma$-matricial algebra.
\end{corollary}

\begin{proof}    By the similarity trick from the start
of this section, we can assume that the
$q_k$ are projections, and $A$ is the closure of
$\sum_k \, q_k A$.  We may then appeal to Theorem \ref{n13} (ii).
  \end{proof}

We now consider a class of
algebras which are a commutative variant of matricial operator
algebras, and are ideals in their bidual.

\begin{proposition}  \label{n9}  Let $A$ be a commutative operator algebra
with no nonzero annihilators in $A$, and possessing
a sequence of nonzero
algebraically minimal idempotents $(q_k)$ with $q_j q_k = 0$ for $j
\neq k$, and $\overline{\sum_k \, A q_k} = A$.   Then $A$ is a
semisimple annihilator algebra with dense socle, and $A$ is an ideal in
its bidual. 
If further the $q_k$ are projections (resp.\ $\sum_k \, q_k = 1$ strictly), and
if $A$  is left essential, then $A \cong c_0$
isometrically (resp.\ $A \cong c_0$ isomorphically).
\end{proposition}

\begin{proof} 
If $x \in J(A)$, the Jacobson radical, then $x q_k \in J(A) \cap
 \Cdb q_k = 0$ for all
$k$, since $J(A)$ contains no nontrivial idempotents.  Thus $x A = 0$
and $x = 0$, and so  $A$ is semisimple.
If $J$ is a closed ideal in $A$ with $q_k J \neq (0)$ for all $k$,
then $q_k \in J$ since $q_k J \subset J \cap \Cdb q_k$.  Thus
$\sum_k \, A q_k \subset J$ and $J = A$.  So $A$ is
an annihilator algebra. 
 The $q_k A$ are minimal ideals and so $A$ has dense
socle. 
By Proposition \ref{n1}, 
$A$ is an ideal in its bidual.   We
leave the other assertions as an exercise.  \end{proof}

{\sc Remark.}  If in addition to the conditions
 in the first sentence of Proposition  \ref{n9}, $\Delta(A)$ acts nondegenerately
on $A$ then $A$ is $\Delta$-dual and nc-discrete (using Proposition  \ref{elli}).

\begin{example}  \label{n10}  The following example illustrates
the distinction between the condition $\overline{\sum_k \, q_k A}
= \overline{\sum_k \, A q_k} = A$,  and the condition
$\sum_k \, q_k  = 1$ strictly (the latter defining algebras
isomorphic or similar to a $\sigma$-matricial algebra by Corollary   \ref{n12}).
 Inside $B = M_2 \oplus^\infty M_2 \oplus^\infty \cdots$, we
consider idempotents $q_{2k} = 0 \oplus \cdots \oplus 0 \oplus e_k
\oplus 0 \oplus \cdots$ and $q_{2k+1}  = 0 \oplus \cdots \oplus 0
\oplus f_k \oplus 0 \oplus \cdots$, where $e_k, f_k$ are idempotents
in $M_2$ with $e_k f_k = 0, e_k + f_k = I_2$, and $\Vert e_k \Vert,
\Vert   f_k \Vert  \to \infty$.   For example, consider the rank one
operators $e_k = [1 : 1] \otimes [-k : k+1]$ and $f_k = [(k+1)/k :
1] \otimes [k : -k]$ in $M_2$. Let $A$ be the closure of the span of
these idempotents $(q_{k})$, which has cai, and may be viewed as a
subalgebra of $\Kdb(\ell^2)$.  The algebra $A$ is of the type
discussed in the last result, and the remark after it.  Indeed
it is a `dual Banach algebra' in the sense of Kaplansky.  However
$A$ is not isomorphic to a
$\sigma$-matricial algebra, indeed is not isomorphic to $c_0$, since
the algebraically minimal idempotents in $A$ are not
uniformly bounded, whereas they are in $c_0$.
\end{example}

Let $A$ be any operator algebra.  If $e_1, \cdots , e_n$ are
algebraically minimal  projections in $A$, set $e = e_1 e_2 \cdots
e_n$.  Then $e^2 = t e_1 e_2 \cdots e_n = t e$, for some $t$ with
$|t| \leq  1$.  Note that $t=0$ iff $e$ is nilpotent, whereas if $t
\neq 0$ then 
$\frac{1}{t} e$ is an algebraically minimal idempotent. The
set $E$ of linear combinations of such products is a $*$-subalgebra
of $\Delta(A)$, and so $\bar{E}$ is the $C^*$-subalgebra $B$ of
$\Delta(A)$ generated by the algebraically minimal projections in
$A$.  Note that the sum of all minimal right ideals of $B$ is dense
in $B$, so that $B$ is an
 annihilator $C^*$-algebra. 
  We write $B$ as $\Delta$-soc$(A)$.

For an operator algebra $A$, define the {\em r-socle} r-soc$(A)$ to be the closure
of the sum of r-ideals of the form $eA$ for algebraically minimal
 projections $e$. This is an r-ideal, with support
projection $f$ equal to  the `join' of all the algebraically minimal projections.
Thus r-soc$(A) = f
A^{**} \cap A$.  Note that $f \in M(B) = B^{**}$ where $B = \Delta$-soc$(A)$.
 Similarly, $\ell$-soc$(A) = A^{**} f \cap A$ is the closure of the
sum of $\ell$-ideals of the form $Ae$ for such $e$, and h-soc$(A)$
is the matching HSA $fA^{**} f \cap A$. We say that the i-socle
exists if r-soc$(A) = \ell$-soc$(A)$, an approximately unital ideal,
which also equals h-soc$(A)$ in this case.  Note that r-soc$(A) \cap
J(A) = (0)$ by p.\ 671 of \cite{Pal}, hence h-soc$(A) \cap J(A) =
(0)$.

\begin{theorem}  \label{n133}  Let $A$ be a
semiprime operator algebra. 
 Then h-soc$(A)$ is a $\sigma$-matricial algebra. \end{theorem}

\begin{proof} We use the notation above.
 Let $B = \Delta$-soc$(A)$, an annihilator
$C^*$-algebra.  Let $D = \, $ h-soc$(A) = fA^{**} f \cap A$,
 an approximately unital
semiprime operator algebra by Proposition \ref{spp}.  Set $J
= fA^{**}  \cap A$. 
Let  $(f_k)_{k \in E}$ be a maximal family of mutually orthogonal
algebraically minimal projections in $D$, and set $e = \sum_{k \in
E} \, f_k \in M(B)$.  Note that $e \leq f$.  Suppose that  $f \neq e$.
Then $(f-e) g \neq 0$ for some
algebraically minimal projection $g$ in $A$ (or else $(f-e) f = 0$,
which is false). Then
$p =  t (f-e) g (f-e)$ is an algebraically minimal projection for
some $t > 0$, which lies in $B$ since $f,e \in M(B)$.   Thus $p \in
A \cap f A f \subset D$, contradicting the maximality of the family.
So $f = \sum_{k \in E} \, f_k$, and $J = \oplus^c_{k \in E} \, f_k
A$ by the argument that  (v) implies (ii) in Theorem \ref{n13} ($f \in
(\oplus^c_{k \in E} \, f_k A)^{\perp \perp}$ so $J^{\perp \perp} = f
A^{**} = (\oplus^c_{k \in E} \, f_k A)^{\perp \perp}$).
 The partial sums of $\sum_{k \in E}
\, f_k$ are a positive left  cai for $J$, so they are a cai for $D$
\cite{BHN}. So $\sum_{k \in E} \, f_k = f$ strictly on $D$. By
Theorem \ref{n13}, $D$ is a $\sigma$-matricial algebra.
\end{proof}

 If $A$ is a $\sigma$-matricial algebra, then the r-ideals, $\ell$-ideals,
and HSA's of $A$ are of a very nice form: 

\begin{proposition}  \label{n17}  If $A$ is a $\sigma$-matricial algebra,
then for every r-ideal (resp.\ $\ell$-ideal) $J$ of $A$, 
there exist mutually orthogonal
algebraically minimal projections $(f_k)_{k \in I}$ in $A$ with
$\sum_k \, f_k = 1$ strictly on $A$, and $J = \oplus^c_{k \in E} \,
f_k A$ (resp.\ $J = \oplus^r_{k \in E} \, A f_k$), for some set $E
\subset I$. 
\end{proposition}

\begin{proof}
By Proposition \ref{elli}, every r-ideal $J$ equals $f A$ for a
projection $f \in M(A)$. Then $D = f A f$ is the HSA corresponding
to $J$, and $A f$ is the corresponding $\ell$-ideal. 
 We follow the proof of Theorem  \ref{n133}.   Let  $(f_k)_{k \in E}$ be
 as in that proof, then  $e = \sum_{k \in E} \, f_k \in
M(\Delta(A)) = \Delta(M(A))$.  If $f \neq e$ then $f-e \in M(A)$, so
$f-e$ dominates a nonzero algebraically minimal projection in $D$,
producing a contradiction. So $f = \sum_{k \in E} \, f_k$, and $J =
\oplus^c_{k \in E} \, f_k A$ as before.  A similar argument shows
that $Af =  \oplus^r_{k \in E} \, A f_k$.

 By a maximality argument (similar to  the proof of (v) implies
(ii) in Theorem \ref{n13}), we can enlarge $(f_k)_{k \in E}$ to a
set $(f_k)$ in $A$ with $A = \oplus^c_k f_k A$.  Then $\sum_k \, f_k
= 1$ strictly on $A$.  
\end{proof}

\begin{corollary}  \label{n17n}  Every HSA in a $1$-matricial algebra (resp.\ in a
$\sigma$-matricial algebra)  is a $1$-matricial algebra (resp.\ a
$\sigma$-matricial algebra).
\end{corollary}

\begin{proof}    We may assume that $A$ is a $1$-matricial algebra.
Continuing the proof of Proposition \ref{n17}: as in the proof of
Theorem \ref{n133}, $\sum_{k \in E} \, f_k = f$ strictly on $D$. By
Proposition \ref{spp}, $D$ is topologically simple. By Theorem
\ref{chgk1}, $D$ is a $1$-matricial algebra.
\end{proof}

\section{CHARACTERIZATIONS
OF $C^*$-ALGEBRAS OF COMPACT OPERATORS}

An interesting question is whether every approximately unital operator algebra
with the property that
all closed right ideals have a  left cai (and/or similarly for
left ideals), is a $C^*$-algebra.   The following is a partial result
along these lines:

\begin{theorem}  \label{who}  Let $A$ be a semiprime
 approximately unital operator algebra.  The following are equivalent:
 \begin{itemize} \item [(i)]  Every minimal
right ideal of $A$ has a left cai (or equivalently by Lemma {\rm
\ref{aus}}, equals $pA$ for a projection $p \in A$).
\item [(ii)]  Every
algebraically minimal idempotent in $A$ has range projection in $A$.
\end{itemize}
If either of these hold, and if $A$ has dense socle, 
 then $A$ is completely
isometrically isomorphic to an annihilator $C^*$-algebra.
\end{theorem}

\begin{proof}  (i) $\Rightarrow$ (ii) \  If 
$e$ is an algebraically minimal idempotent in $A$, then $eA$ is a
minimal right ideal, hence an r-ideal by (i).  By Lemma \ref{aus},
$eA = pA$ for a projection $p \in A$.  We have $p e = e, ep = p$,
which forces $p$ to be the range projection of $e$.

(ii)  $\Rightarrow$ (i) \ Every minimal
right ideal equals $eA$ for an algebraically minimal idempotent $e$.
 The range projection $p$ of $e$ satisfies
$p e = e, ep = p$, so that $eA = pA$.

Suppose that these hold, and $A$ has dense socle.  We will argue
that $A$ satisfies Theorem \ref{n13} (iv), hence is a
$\sigma$-matricial algebra.  As in the proof of Theorem \ref{n13cc},
$A$ is a semisimple
 modular annihilator algebra, and $\Delta(A)$ is a
nonzero  annihilator $C^*$-algebra. 
Let $e_\Delta$ be the identity of $\Delta(A)^{**}$, and set 
$J = e_\Delta A^{**} \cap A$, a closed right ideal in $A$.  If $e$
is an algebraically minimal idempotent not in $J$, and if $q$ is its
range projection, which is in $A$, then $q$ is not in $J$, or else
$e  = q e$ would be in $J$. However $q \leq e_\Delta$ since $q \in
\Delta(A)$, so that $q \in J$. Thus every algebraically minimal
idempotent in $A$ is in $J$. Hence $J$ contains the socle, so that
$J = A$ and $e_\Delta = 1$. Therefore $A$ is $\Delta$-dual. If $p$
is a $*$-minimal projection in $A$, then $pA$ contains a nonzero
algebraically minimal idempotent $e \in A$, by \cite[Theorem 8.4.5
(h)]{Pal}.  By the hypothesis, $eA = fA$ for a projection $f \in A$.
Clearly $f$ is algebraically minimal, and $f \in eA \subset pA$, so
that $pf = f = fp$.  Thus $p = f$ is algebraically minimal. 
By Theorem \ref{n13} (iv), $A$ is a $\sigma$-matricial algebra.

It remains to show that every $1$-matricial algebra $A$ satisfying the given
condition,  is a $C^*$-algebra of compact operators. To this end,
 fix integers $i \neq j$
and let $x = q_i + q_j + T_{ij} + T_{ji} \in (q_i + q_j) A (q_i +
q_j)$.  Then $x A = eA$ for a projection $e \in A$,
since $x$ is a scalar multiple of an algebraically minimal idempotent. 
We have $(q_i + q_j) e = e = e (q_i + q_j)$.
Suppose that the $i$-$j$ entry of $e$ is zero, which forces $e$ to
be $q_i, q_j,$ or $q_i + q_j$.
  Let $u = T_{ij}
+ T_{ji}$, then $u x = x$. There exists $(a_n) \subset A$ with $x
a_n \to e$, and it follows that $u e = e$, which is false.  This
contradiction shows that the $i$-$j$ entry of $e$ is nonzero.  By
the proof of Proposition \ref{n99}, $T_i T_i^* = T_j T_j^*$.  Since $i, j$
were arbitrary, it follows as in Proposition \ref{n99} that $A$ is
selfadjoint, and $A \cong \Kdb(H)$ for a Hilbert space $H$. 
\end{proof}

We say that a left ideal  in $A$ is {\em $A$-complemented} if it is the
range of a bounded idempotent left $A$-module map.

\begin{lemma} \label{hihoho} Let $A$ be an operator algebra with a bounded
right approximate identity.
 \begin{itemize}
\item [{\rm (1)}]   Every $A$-complemented
closed left ideal $J$ in $A$ has a right bai, and also a nonzero right
annihilator.   Indeed $J = Ae$ for some idempotent $e \in A^{**}$.
\item [{\rm (2)}]  If every closed left ideal in $A$ is $A$-complemented, then $A$
is a semiprime right annihilator algebra.
\item [{\rm (3)}]  If $A$ has a bai, and if
$J$ is a two-sided ideal in $A$ which is both right and left
complemented, then $J = eA$ for an idempotent $e$ in the center of
$M(A)$.  Also, $J$ has a bai.
\item [{\rm (4)}] If $A$ has a right cai, then every contractively $A$-complemented
closed left ideal $J$ in $A$ has a right  cai, and the $e$ in {\rm (1)}
may be chosen to be a contractive projection in $M(A)$.
\end{itemize}
\end{lemma}

\begin{proof} (1) \ Let $P : A \to J$ be the projection, with $\Vert P \Vert \leq
K'$. If $(e_t)$ is the right bai, then $x P(e_t) = P(x e_t) \to P(x) = x$ for $x \in J$.
Thus $J$ has a right bai.  There is a weak* convergent subnet
$P(e_{t_\mu}) \to r$ weak* in $A^{**}$. Then $ar = \; $w*-lim$_\mu
\, P(ae_{t_\mu}) = P(a) \in J$ for all $a \in A$. So $r \in RM(A)
\cap J^{\perp \perp}$.  Also, $xr = x$ for $x \in J$. Hence 
$J = Ar$, and $\Vert r \Vert
\leq K K'$, if $K$ is a bound for the right bai. If $Ar A = (0)$ then
$ar^2 = ar = 0$ for all $a \in A$, so that $r = 0$. Thus if $J$ is
nontrivial then $J$ has a nonzero right annihilator. Also, $J^{\perp
\perp}$ is a weak* closed left ideal containing $r$ so $A^{**}r =
J^{\perp \perp}$. 

(2)  \ 
In this case, $A$ is a right annihilator algebra, by (1). Also $A$ is
semiprime, since if $J$ were a two-sided ideal with $J^2 = (0)$,
then $ArAr = (0)$, and $r^4 = r = 0$.

(3) \ In this case, $J$ has a right and a left bai, hence a bai \cite{Pal}.

(4) \ This is a slight modification of 
the proof of (1). 
\end{proof}

\begin{corollary}  \label{evc}  Let $A$ be a semisimple approximately unital
 operator algebra such that  every closed left ideal  in  $A$ is contractively
$A$-complemented, or equivalently equals $J = Ap$ for a projection
$p \in M(A)$.  Then   $A$ is completely isometrically isomorphic to
an  annihilator  $C^*$-algebra.
 \end{corollary} \begin{proof} 
 As we said above, $A$ is a right annihilator algebra.  If it is semisimple
 then it has dense socle  
by \cite[Proposition 8.7.2]{Pal}.
 Now apply Theorem \ref{who}.
\end{proof}

{\sc Remarks.} 1) \ This result is related to theorems of Tomiuk and
Alexander (see e.g.\ \cite{Tom,Alex} concerning `complemented Banach
algebras', but the proofs and conclusions are quite different.
\smallskip

\smallskip 

2) \  We imagine that a
semisimple approximately unital
 operator algebra such that  every closed left ideal  in  $A$ is $A$-complemented,
is isomorphic to an annihilator $C^*$-algebra. Certainly in this
case $A$ is a right annihilator algebra by the last lemma. It therefore
has dense socle by \cite[Proposition 8.7.2]{Pal}, and an argument
from \cite{Kap} shows that $A= \overline{\sum_{i} \, A e_i}$ for
mutually orthogonal algebraically minimal idempotents $e_i$ in $A$. 
Also by
 \cite{Kap}, we have a family $\{ S_i : i \in I \}$ of
two-sided closed ideals of $A$, with $S_i S_j =
(0)$ if $i \neq j$, and whose union is dense in $A$.   If we assume
that ideals are {\em uniformly} $A$-complemented, then there is
a constant $K$, and idempotents $f_i$  with $\Vert f_i \Vert \leq K$
and $S_i = A f_i$ for each $i$.   For any finite $J
\subset I$ we have $A (\sum_{i \in J} \, f_i)$ is complemented too,
so that $\Vert \sum_{i \in J} \, f_i \Vert \leq K$.  We may thus use
a similarity as at the start of Section 4 to reduce to the case that $f_i$ are projections.  The
canonical  map $\oplus^f_{i \in I} \, S_i \to A$ is an isometric
homomorphism with respect to the $\infty$-norm on the direct sum, so
that $A \cong \oplus^0_{i \in I} \, S_i$.  Thus if we are assuming
that closed ideals are uniformly complemented then
 we have reduced the question
 to the case that $A$ is a topologically simple annihilator algebra.

\bigskip

We now give a characterization amongst the $C^*$-algebras, of
$C^*$-algebras consisting of compact operators. There are many such
characterizations in the literature, however we have not seen the
one below, in terms of the following notions introduced by Hamana.
If $X$ contains a subspace $E$ then we say that $X$ is an {\em
essential extension} (resp.\ {\em rigid extension}) of $E$ if any
complete contraction with domain $X$ (resp.\ from $X$ to $X$) is
completely isometric (resp.\ is the identity map) if it is
completely isometric (resp.\ is the identity map) on $E$.  If $X$ is
injective then it turns out that it is rigid iff it is essential,
and in this case we say $X$ is an {\em injective envelope} of $E$,
and write $X = I(E)$.   See e.g.\ 4.2.3 in \cite{BLM}, or the works
of Hamana; or \cite{PnWEP} for some related topics.

\begin{theorem} \label{subcom} If $A$ is
a $C^*$-algebra,  the following are equivalent: \begin{itemize}
\item [{\rm (i)}]  $A$ is an annihilator $C^*$-algebra.
\item [{\rm (ii)}] $A^{**}$ is an essential extension of $A$.
\item [{\rm (iii)}] $A^{**}$ is an injective envelope of $A$.
\item [{\rm (iv)}] $I(A^{**})$ is an injective envelope of $A$.
\item [{\rm (v)}]
  Every surjective complete isometry $T : A^{**} \to
A^{**}$ maps $A$ onto $A$. 
\item [{\rm (vi)}]  $A$ is nuclear and $A^{**}$ is a rigid extension of $A$.
  \end{itemize}
\end{theorem}

\begin{proof} Clearly (i) $\Rightarrow$ (iii) (from e.g.\ \cite[Lemma 3.1 (ii)]{HamC}),
and (iii) $\Rightarrow$ (iv) $\Rightarrow$ (ii) by standard diagram chasing.
Since $A$ is nuclear iff $A^{**}$ is
injective, we have (vi) iff (iii).
 
Item  (ii) implies that the normal extension of every faithful
$*$-representation of $A$ is faithful on $A^{**}$. This implies that
$A^{**}$ is injective, the latter since $\pi_a(A)'' \cong
\oplus^\infty_i \; B(H_i)$ is injective (see \cite[Lemma
4.3.8]{Ped}), where $\pi_a$ is the atomic representation of $A$.  So
(ii) $\Leftrightarrow$ (iii). Moreover, in the notation above, these
imply that $\pi_a(A)''$ is an injective envelope of $A$, and hence
$\pi_a(A)$ contains $\oplus^0_i \; K(H_i)$ by \cite[Lemma 3.1
(iii)]{HamC}. Thus $A$ has a subalgebra $B$ with $\pi_a(B) =
\oplus^0_i \; K(H_i)$, and $\tilde{\pi}_a(B^{\perp \perp}) =
\pi_a(A)''$.  Hence  $B^{\perp \perp} = A^{**}$, and so $A = B$. So
(ii) $\Rightarrow$ (i).

From e.g.\ \cite[Proposition 3.3]{Sharma}, (i) $\Rightarrow$ (v).
 Conversely, if $p$ is a projection in $A^{**}$ then $u = 1 - 2p$ is
unitary, and so $(1 - 2p) A \subset A$ if (v) holds.  Indeed clearly
$p \in M(A)$, so that $A^{**} \subset M(A)$. Thus $A$ is an ideal in
$A^{**}$, which implies (i)  \cite{HWW}.
\end{proof}

{\sc Remark.}  We are not sure if  in (vi) one may drop the
nuclearity condition. By standard diagram chasing, for
nonselfadjoint algebras (or operator spaces), (ii) is equivalent to
(iv), and to $A^{**} \subset I(A)$ unitally.  Also (i) $\Rightarrow$
(ii) for nonselfadjoint algebras of compact operators, indeed if $A$
is an operator algebra with cai, which is a left or right ideal in
its bidual, then we have
 (ii) (since $LM(A)$ and $RM(A)$ may be viewed in $I(A)$, see e.g.\
 \cite[Chapter 4]{BLM}),
 and also (v) (by \cite[Proposition 3.3]{Sharma}).
 Also (ii) implies that $A^{**}$ is a rigid
extension (since $I(A)$ is), and this works for operator spaces 
too.  It is easy to see that $A^{**}$ being a rigid extension of
  an operator space $A$, implies that every
surjective complete isometry $T : A^{**} \to A^{**}$ is weak*
continuous (a property enjoyed by all $C^*$-algebras).  To see this,
let $\widetilde{T_{\vert A}} : A^{**} \to A^{**}$ be the weak*
continuous extension of $T_{\vert A}$, then $T^{-1} \circ
\widetilde{T_{\vert A}} = I_{A^{**}}$ by rigidity. 

In  \cite{HamR,HamC}, Hamana defines the notion of a {\em regular}
extension of a  $C^*$-algebra.   It is not hard to see that
$A^{**}$ is an essential extension of $A$ iff it is a regular
extension. This uses the fact that (ii) is equivalent to $A^{**}
\subset I(A)$, and the fact that the regular monotone completion of
$A$ from \cite{HamR}, resides inside $I(A)$ (see \cite{Sharma2}
for more details if needed). 

\section{ONE-SIDED IDEALS IN TENSOR PRODUCTS OF OPERATOR ALGEBRAS}

Amongst other things, in this section we extend several known
results about the Haagerup tensor products of $C^*$-algebras (mainly
from \cite{Bgeo,BZ}), to general operator algebras, and give some
applications.  For example, we investigate the one-sided $M$-ideal
structure of the  Haagerup tensor products of nonselfadjoint
operator algebras.

We will write $M\otimes^{\sigma h} N$ for the {\em $\sigma$-Haagerup
tensor product} (see e.g.\ \cite{ERc,EK,BM,BLM}). We will repeatedly
use the fact that for operator spaces $X$ and $Y$, we have $(X
\otimes_h Y)^{**} \cong X^{**} \otimes^{\sigma h} Y^{**}$ (see e.g.\
1.6.8 in \cite{BLM}).  We recall from \cite[Section 3]{BM} that the
Haagerup tensor product and $\sigma$-Haagerup tensor product of
unital operator algebras is a unital operator space (in the sense of
\cite{BM}), and also is a unital Banach algebra. We write ${\rm
Her}(D)$ for the hermitian elements in a unital space $D$ (recall
that $h$ is hermitian iff $\varphi(h) \in \Rdb$ for all $\varphi \in
{\rm Ball}(D^*)$ with $\varphi(1) = 1$).

\begin{lemma} \label{heha}
If $A$ and $B$ are unital operator spaces then ${\rm
Her}(A\otimes_{h} B) =A_{\rm sa}\otimes 1+1\otimes B_{\rm sa}$ and
$\Delta(A\otimes_{h} B)= \Delta(A) \otimes 1+1\otimes \Delta(B)$.
Similarly, if  $M$ and $N$ are unital dual operator algebras, then
${\rm Her}(M\otimes^{\sigma h} N) =M_{\rm sa}\otimes 1+1\otimes
N_{\rm sa}$ and $\Delta(M\otimes^{\sigma h} N)= \Delta(M) \otimes
1+1\otimes \Delta(N)$.
\end{lemma}

\begin{proof}
If $A$ and $B$ are unital operator spaces then $A\otimes_{h} B$ is a
unital operator space (see \cite{BM}), and ${\rm Her}(A\otimes_{h}
B) \subset {\rm Her}(C^*(A) \otimes_{h} C^*(B)).$ By  a result in
\cite{Bgeo}, it follows that if $u \in {\rm Her}(A\otimes_{h} B)$
then there exist $h \in C^*(A)_{\rm sa}, k \in C^*(B)_{\rm sa}$ such
that $u = h \otimes 1+1\otimes k$.    It is easy to see that this
forces $h \in A, k \in B$. For example if $\varphi$ is a functional
in $A^\perp$ then $0 = (\varphi \otimes I_B)(u) = \varphi(h) 1$, so
that $h \in (A^\perp)_\perp = A$. Conversely, it is obvious that
$A_{\rm sa}\otimes 1+1\otimes B_{\rm sa} \subset {\rm
Her}(A\otimes_{h} B)$.  Indeed the canonical maps from $A$ and $B$
into $A\otimes_{h} B$ must take hermitians to hermitians.  This
gives the first result, and taking spans gives the second.

Now let $M$ and $N$ be unital dual operator algebras. Again it is
obvious that $M_{\rm sa}\otimes 1+1\otimes N_{\rm sa} \subset {\rm
Her}(M\otimes^{\sigma h} N)$. For the other direction, we may assume
that $M = N$ by the trick of letting
 $R=M\oplus N$.  It is easy to argue that
$M\otimes ^{\sigma h}N \subset R\otimes ^{\sigma h}R$, since $M$ and $N$ are
appropriately complemented in $R$. 
If $W^*_{\rm max}(M)$ is the `maximal  von Neumann algebra'
generated by $M$, then by Theorem 3.1 (1) of \cite{BM}  we have
 $M\otimes^{\sigma h} M \subset W^*_{\rm max}(M) \otimes^{\sigma h}
W^*_{\rm max}(M)$.  So (again using the trick in the first paragraph
of our proof) we may assume that $M$ is a von Neumann algebra. By a
result of Effros and Kishimoto \cite[Theorem 2.5]{EK}, ${\rm
Her}(M\otimes^{\sigma h} M)$ equals $${\rm
Her}(CB_{M'}(B(H)))\subset {\rm
Her}(CB(B(H))) = \{h\otimes 1+1\otimes
k:\;
 h,k\in B(H)_{\rm sa}\},$$
  the latter by a result of Sinclair and Sakai (see
e.g.\ \cite[Lemma 4.3]{BSZ}). By a small modification of the
argument in the first paragraph of our proof it follows that
 $h,k \in M$.   The final result again follows by taking the span. 
\end{proof}

 \begin{theorem} \label{ha}
Let $M$ and $N$ be unital dual
operator algebras.  If $\Delta(M)$ is not one-dimensional then $\Delta(M) \cong
\mathcal{A}_\ell(M\otimes^{\sigma h}N)$.  If $\Delta(N)$ is not one-dimensional then
$\Delta(N) \cong \mathcal{A}_r(M\otimes^{\sigma h}N).$
If $\Delta(M)$ and $\Delta(N)$ are one-dimensional then
$$\mathcal{A}_\ell(M\otimes^{\sigma h}N) = \mathcal{A}_r(M\otimes^{\sigma h}N) =
\Cdb I .$$    \end{theorem} 
\begin{proof}   We just prove the first and the last assertions.  Let $M$ and $N$ be unital dual
operator algebras, and let $X = M\otimes^{\sigma h} N$.  The map
$\theta : {\mathcal A}_{\ell}(X) \to X$ defined by $\theta(T) =
T(1)$ is a unital complete isometry (see the end of the notes
section for 4.5 in \cite{BLM}).   Hence, by  \cite[Corollary
1.3.8]{BLM} and Lemma \ref{heha}, it maps into $\Delta(X) =
\Delta(M) \otimes 1 + 1 \otimes \Delta(N)$.    The last assertion is now clear.
For the first, if we can show that
Ran$(\theta) \subset \Delta(M) \otimes 1$, then we will be done.
There is a copy of $\Delta(M)$ in ${\mathcal A}_{\ell}(X)$ via the
embedding $a \mapsto L_{a \otimes 1}$, 
and this is a $C^*$-subalgebra. Note that $\theta$ restricts to a
$*$-homomorphism from this $C^*$-subalgebra into the free product $M
* N$ discussed in \cite{BM}. Let $T \in {\mathcal A}_{\ell}(X)_{\rm sa}$,
then $\theta(T) \in X_{\rm sa}$.
  By Lemma \ref{heha}, $T(1 \otimes 1) = h \otimes 1 + 1 \otimes k$, with $h \in \Delta(M)_{\rm sa},
k \in \Delta(N)_{\rm sa}$. It suffices to show that $\theta(T - L_{h
\otimes 1}) = 1 \otimes k \in \Delta(M) \otimes 1$. So let $S = T -
L_{h
\otimes 1}$.
By \cite[Proposition 1.3.11]{BLM} we have for $a \in \Delta(M)_{\rm
sa}$ that
$$S(a \otimes 1) = \theta(S L_{a \otimes 1}) = \theta(S) * (a \otimes 1) =
(1 \otimes k) * (a \otimes 1) .$$ The involution in $M * N$, applied
to the last product, yields $a * k = a \otimes k \in M
\otimes^{\sigma h} N$.  Hence $$S(a \otimes 1) \in
\Delta(M\otimes^{\sigma h} N) =  \Delta(M) \otimes 1 + 1 \otimes
\Delta(N) \subset \Delta(M) \otimes \Delta(N).$$ Since left and
right multipliers of an operator space automatically commute, we have that
 $\rho(\Delta(N))$ commutes with $S$, where $\rho: \Delta(N) \to
{\mathcal A}_r(M \otimes^{\sigma h} N)$ is the canonical injective
$*$-homomorphism. Thus for $b \in \Delta(N)$ we have
\[
 S(a \otimes b) = S(\rho(b)(a \otimes 1)) = \rho(b)(S(a \otimes 1))
 = S(a \otimes 1) (1 \otimes b) \in \Delta(M) \otimes \Delta(N) .
 \]
 By linearity this is true for any $a \in
\Delta(M)$ too.  It follows that $\Delta(M) \otimes_{h} \Delta(N)$
is a subspace of $M\otimes^{\sigma h} N$ which is invariant under
$S$.  Since  $S$  is selfadjoint, it follows from \cite[Proposition
5.2]{BZ} that the restriction of $S$ to $\Delta(M) \otimes_{h}
\Delta(N)$ is adjointable, and selfadjoint. Hence by \cite[Theorem
5.42]{BZ} we have that there exists an $m \in \Delta(M)$ with $S(1
\otimes 1) = m \otimes 1 = 1 \otimes k$.  Thus $1 \otimes k \in
\Delta(M) \otimes 1$ as desired.
\end{proof}

\begin{corollary}  \label{chtp}
Let $A$ and $B$ be approximately unital operator
algebras.  If $\Delta(A^{**})$ is not one dimensional
then    $\Delta(M(A))\cong {\mathcal A}_{\ell}(A \otimes_h B)$.
If $\Delta(B^{**})$ is not one dimensional
then   $\Delta(M(B))
\cong {\mathcal A}_{r}(A \otimes_h B)$.  If
$\Delta(A^{**})$ and $\Delta(B^{**})$ are one dimensional
then ${\mathcal A}_{\ell}(A \otimes_h B)
= {\mathcal A}_{r}(A \otimes_h B) = \Cdb I$.
  \end{corollary}
\begin{proof} We just prove the first and last relations.
Let $\rho: \Delta(LM(A))\to {\mathcal A}_{\ell}(A \otimes_{h} B)$
be the injective $*$-homomorphism given by $S\mapsto S\otimes I_{B}$. If $T\in
{\mathcal A}_{\ell}(A \otimes_{h} B)_{\rm sa}$, then by Proposition 5.16
 from \cite{BZ}, we have
$T^{**}\in {\mathcal A}_{\ell}(A^{**} \otimes^{\sigma h} B^{**})_{\rm sa}$. By the last
theorem, $T^{**}(a\otimes b)= T(a\otimes b)= L_{h\otimes 1}(a\otimes b)$,
 for some $h \in A ^{**}_{\rm sa}$ and for all $a\in A$, $b \in B$.
Since $T(a\otimes b)$ is in $A \otimes_{h} B$, so is $L_{h\otimes
1}(a\otimes b)$ for all $a\in A$,
 $b \in B$. Also $L_{h\otimes 1}(a\otimes 1)=ha\otimes 1 \in A \otimes_h B$ for all $a\in A$.
So $ha \in A$ for all $a \in A$. Thus $L_{h}\in \Delta(LM(A)_{\rm
sa})$. This shows that $\rho$ is surjective, since selfadjoint
elements span ${\mathcal A}_{\ell}(A \otimes_h B)$. Thus
$\Delta(LM(A))\cong {\mathcal A}_{\ell}(A \otimes_h B)$.  By
the proof of \cite[Proposition 5.1]{BHN}, we have
$\Delta(LM(A))=\Delta(M(A))$.  This proves the first relation.  If
$\Delta(A^{**})$ and $\Delta(B^{**})$ are one dimensional, then so is $\Delta(M(A))$, and so
is ${\mathcal A}_{\ell}(A^{**} \otimes^{\sigma h} B^{**})$, by the
theorem.  Hence the $T$ above is in $\Cdb I$, and this proves the
last assertion.     \end{proof}

{\sc Remark.}  For $A, B, M, N$ as in the last results, it is probably true that
we have $\Delta(M(A))
\cong {\mathcal A}_{\ell}(A \otimes_h B)$,
and similarly that
$\Delta(M) \cong {\mathcal A}_{\ell}(M \otimes^{\sigma h} N)$,
 if $A$ and $M$ are not one-dimensional, with no other restrictions.
We are able to prove this if $B = N$ is a finite dimensional $C^*$-algebra.

\bigskip

The following is a complement to  \cite[Theorem 5.38]{BZ}:

 \begin{theorem} \label{mihtp}   Let $A$ and $B$ be approximately unital operator
algebras, and suppose that
 $\Delta(A^{**})$ is not one-dimensional.  Then the
right $M$-ideals (resp.\ right $M$-summands) in $A \otimes_{h} B$ are
precisely the subspaces of the form $J \otimes_{h} B$, where $J$ is
a closed right ideal in $A$ having a left cai (resp.\ having form
$eA$ for a projection $e \in M(A)$).
\end{theorem}
 \begin{proof}  The summand case follows immediately from Corollary
 \ref{chtp}.
The one direction of the $M$-ideal case is \cite[Theorem 5.38]{BZ}.
For the other, suppose that $I$ is a right $M$-ideal in $A
\otimes_{h} B$. View
 $(A\otimes_h B)^{**} =
  A^{**}\otimes^{\sigma h}B^{**}$.  Then $I^{\perp\perp}$
is a right $M$-summand in $\A^{**} \otimes^{\sigma h} \B^{**}$. By
Theorem \ref{ha} we have  $I^{\perp \perp} = e A^{**}\otimes^{\sigma
h}B^{**}$ for a projection $e \in A^{**}$. 
 Let $J=e A^{**} \cap A$, a closed right ideal in $A$.
 We claim that $I=J\otimes_h B$.
 Since $I= I^{\perp\perp}\cap(A\otimes_h B)$, we need to show that
$(e A^{**}\otimes^{\sigma h}B^{**})\cap (A \otimes_h B) = (e A^{**}\cap A)\otimes_h B$.
By injectivity of $\otimes_h$, it is clear that
 $(e A^{**}\cap A)\otimes_h B \subset (e A^{**}\otimes^{\sigma h}B^{**})\cap (A \otimes_h B)$.
 For the other containment, we let $u \in (e A^{**}\otimes^{\sigma h}B^{**})\cap (A \otimes_h B)$,
and use a slice map argument.
By \cite[Corollary 4.8]{Smi}, we need to show that for all $\psi\in
B^*$, $(1\otimes \psi)(u)\in e A^{**}\cap A=J$. Let $\psi\in B^*$,
then $\langle \tilde{u} , 1\otimes \psi \rangle = (1\otimes \psi)(u)
\in A$,
 where $\tilde{u}$
is $u$ regarded as an element in $A^{**}\otimes^{\sigma h}B^{**}$.
Since $u\in e A^{**}\otimes^{\sigma h}B^{**}$, we have
$\langle \tilde{u} , 1\otimes \psi \rangle \in e A^{**}$. So $(1\otimes \psi)(u)\in e
A^{**}\cap A=J$, and so  $u\in J\otimes_h B$ as desired.

Next we show that $J$ has
a left cai. It is clear that  $J^{\perp\perp}=\overline{J}^{w^*} \subset e A^{**}$.
Suppose that there is $x\in e A^{**}$ such that $x \notin J^{\perp\perp}$.
Then there exists  $\phi\in J^{\perp}$ such
that $x(\phi)\neq 0$. Since $I=J\otimes_h B$ and $\phi\in J^{\perp}$,
we have $\phi\otimes \psi \in I^{\perp}$ for all states $\psi$ on $B$.
So $I^{\perp\perp}$ annihilates $\phi\otimes \psi$, and
 in particular $0=(x\otimes 1)(\phi\otimes \psi)=x(\phi)$, a contradiction.
Hence $J^{\perp\perp}=e A^{**}$, and it follows from basic principles
about approximate identities
that  $J$ has a left cai.
\end{proof}

\begin{theorem}  \label{ihtp}
Let $M$ and $N$ be unital (resp.\ unital dual) operator algebras,
with neither $M$ nor $N$ equal to $\mathbb{C}$. Then the operator
space centralizer algebra $Z(M \otimes_h N)$ (resp.\ $Z(M
\otimes^{\sigma h} N)$) (see {\rm \cite[Chapter 7]{BZ}}) is
one-dimensional.
  \end{theorem}

 \begin{proof}  First we consider the dual case.  If $\Delta(M)$ and $\Delta(N)$ are both
one-dimensional then $Z(M \otimes^{\sigma h} N) \subset
{\mathcal A}_{\ell}(M \otimes^{\sigma h} N)  = \Cdb I$, and we are done.
If $\Delta(M)$ and $\Delta(N)$ are both not one-dimensional, let
$P$ be a projection in $Z(M \otimes^{\sigma h} N)$. 
 By the theorem,
 $Px=ex=xf$, for all $x\in M \otimes^{\sigma h} N$, for some projections $e\in M$ and $f\in N$.
 Then \[ e^{\perp} \otimes f=e^{\perp}\otimes f f= P(e^{\perp} \otimes f) =
e e^{\perp} \otimes f=0, \]
which implies that either $e^{\perp}=0$ or $f=0$. Hence $P=0$ or $P=I$. So $Z(M \otimes^{\sigma h} N
)$ is
a von Neumann algebra with only
trivial projections, hence it is trivial.

Suppose that $\Delta(N)$ is one-dimensional, but $\Delta(M)$ is not.
Again it suffices to show that any projection $P \in Z(M
\otimes^{\sigma h} N)$ is trivial. By Theorem \ref{ha}, $P$ is of
the form $Px=ex$ for a projection $e\in M$. Assume that $e$ is not
$0$ or $1$. If $D = {\rm Span}\{ e,1-e \}$, and $X$ is the copy of
$D \otimes N$ in $M \otimes^{\sigma h} N$, then $P$ leaves $X$
invariant.  Note that $X = D \otimes_h N$, since $\otimes_h$ is
known to be completely isometrically contained in $\otimes^{\sigma
h}$ (see \cite{ERc}).    Hence by Section 5.2 in \cite{BZ} we have
that the restriction of $P$ to $X$ is in ${\mathcal A}_\ell(X) \cap
{\mathcal A}_r(X) = Z(X)$.  Thus we may assume without loss of
generality that $M = D = \ell^\infty_2$, and $P$ is left
multiplication by $e_1$, where $\{ e_1, e_2 \}$ is the canonical
basis of $\ell^\infty_2$.   Since $P$ is an $M$-projection,  $\Vert
e_1 \otimes x + e_2  \otimes y \Vert_h = \max \{ \Vert x \Vert ,
\Vert y \Vert \}$, for all $x, y \in N$.  Set $x = 1_N$, and let $y
\in N$ be of norm $1$. Then $\Vert e_1 \otimes 1 + e_2 \otimes y
\Vert_h = 1$.  If we can show that $y \in \Cdb 1_N$ then we will be
done: we will have contradicted the fact that $N$ is not
one-dimensional, hence $e$, and therefore $P$, is trivial.  By the
injectivity of the Haagerup tensor product, we may replace $N$ with
Span$\{ 1,y \}$.  By basic facts about the Haagerup tensor product,
there exist $z_1, z_2 \in \ell^\infty_2$ and $v, w \in N$ with $e_1
\otimes 1 + e_2 \otimes y = z_1 \otimes v + z_2 \otimes w$, and with
$\Vert [z_1 \; z_2 ] \Vert^2 = \Vert v^* v + w^* w \Vert = 1$.
Multiplying by $e_1 \otimes 1$ we see that $z_1(1) v + z_2(1) w =
1$, so that
$$1 \leq (|z_1(1)|^2 + |z_2(1)|^2) \Vert v^* v + w^* w \Vert =
|z_1(1)|^2 + |z_2(1)|^2 \leq \Vert [z_1 \; z_2 ] \Vert^2 = 1 .$$
From basic operator theory, if a pair of contractions have product
$I$, then the one is the adjoint of the other. 
Thus $v, w$, and hence $y$, are in $\Cdb 1$.

A similar argument works if $\Delta(M)$ is one-dimensional, but $\Delta(N)$ is not.

In the `non-dual case', use \cite[Theorem 7.4 (ii)]{BZ}
to see that $Z(M \otimes_h N) \subset Z(M^{**} \otimes^{\sigma h} N^{**}) = \mathbb{C} I$.
\end{proof}

 \begin{corollary}
Let $A$ and $B$ be approximately unital operator
algebras, with neither being one-dimensional. Then $A\otimes_h B$
contains no non-trivial complete $M$-ideals.
\end{corollary}

 \begin{proof}
Suppose that $J$ is a complete $M$-ideal in $A\otimes _h B$.
The complete $M$-projection onto $J^{\perp\perp}$
is in $Z((A\otimes_h B)^{**}) = Z(A^{**}\otimes^{\sigma h}B^{**})$,
and hence is trivial by Theorem \ref{ihtp}.
\end{proof}

{\sc Remark.}  The ideal structure of the Haagerup tensor product of
$C^*$-algebras has been studied in \cite{SSm} and elsewhere.

\begin{proposition}\label{example1}
Let $A$ and $B$ be approximately unital operator algebras with $A$ not a
reflexive Banach space, $B$
finite dimensional and $B\ne \mathbb{C}$. If $A$ is a right ideal in $A^{**}$, then
$A \otimes_{h} B$ is a right $M$-ideal in its second dual,
and it is not a left $M$-ideal in its second dual.
\end{proposition}
\begin{proof}
 Since $A$ is a right $M$-ideal in $A^{**}$,
$A\otimes_{h} B$ is a right $M$-ideal in $A^{**} \otimes_{h} B$ by
\cite[Proposition 5.38]{BZ}. Since $B$ is finite dimensional,
$(A\otimes_{h} B)^{**}= A^{**} \otimes_{h} B$ (see e.g.\
\cite[1.5.9]{BLM}). Hence $A \otimes_{h} B$ is a right $M$-ideal in
its bidual. Suppose that it is also a left $M$-ideal.  Then it is a
complete $M$-ideal in its bidual, and therefore corresponds to a
projection in $Z(A^{(4)}\otimes_{h} B)$. However, the latter is
trivial by Theorem  \ref{ihtp}.  This forces $A \otimes_{h} B$,
and hence $A$,  to be reflexive,  which is a contradiction. So
$A\otimes_{h} B$ is not a left $M$-ideal in its bidual.
\end{proof}

{\sc Remark.}   A similar argument (see \cite{Sharma}) shows that if
$Y$ is a non-reflexive operator space which is a right $M$-ideal in
its bidual and if $X$ is any finite dimensional operator space, then
$Y \otimes_{h} X$ is  a right $M$-ideal in its bidual. Further, if
$Z(Y^{(4)}\otimes_{h} X) \cong \mathbb{C}I$ then $Y \otimes_{h} X$ is
not  a left $M$-ideal in its bidual.

\bigskip

The last few paragraphs, and Corollary  \ref{n5} and Example
\ref{n8}, provide natural examples of spaces which are right but not
left $M$-ideals in their second dual. Their duals will be left but
not right $L$-summands in their second dual, by the next result.  We
refer to \cite{BEZ,BZ} for notation.

\begin{lemma} If an operator space $X$ is a right but not a left $M$-ideal in its second dual,
then $X^*$ is a left but not a right $L$-summand  in its second dual.
\end{lemma}

\begin{proof}
We first remark that a subspace $J$ of  operator space $X$ is a
complete $L$-summand of $X$  if and only if  it is both a  left and
a right $L$-summand.  This follows e.g.\ from the matching statement
for $M$-ideals \cite[Proposition 4.8.4]{BLM}, and the second
`bullet' on p.\ 8 of \cite{BZ}. By \cite[Proposition 2.3]{Sharma},
$X^*$ is a left  $L$-summand in $X^{***}$, via the canonical
projection $i_{X^*} \circ (i_X)^*$. Thus if $X^*$ is both a  left
and a right $L$-summand in its second dual, then $i_{X^*} \circ
(i_X)^*$ is a left $L$-projection by the third `bullet' on p.\ 8 of
\cite{BZ}. Hence by \cite[Proposition 2.3]{Sharma}, $X$ is a left
$M$-ideal in its second dual, a contradiction.
\end{proof}

We end with some remarks complementing some other results in
\cite{Sharma}.

1) \  Theorem 3.4  (i) of \cite{Sharma} can be improved in the case
that $X$ is an approximately unital operator algebra $A$. Theorem
3.4 (i) there, is valid for all one-sided $M$-ideals, both right and
left.   This follows from Proposition \ref{elli} and Proposition
\ref{n95}.  

2) \  Theorem 3.4  (iii) of \cite{Sharma} can also be improved in
the case that $X$ is an operator algebra $A$.
 If $A$ is
an operator algebra with right cai which is a left ideal in
$A^{**}$ (or equivalently, if $A$ is a
left $M$-ideal in its bidual), and if $J$ is a right ideal in $A^{**}$, then
$J A \subset J \cap A$.  Hence if $J \cap A = (0)$ then
$J A  = (0)$.  Thus $J A^{**} = (0)$, and hence $J = (0)$, since
$A^{**}$ has a right identity.   Thus the case $J \cap A = (0)$
will not occur in
the conclusion of Theorem 3.4 (iii) of \cite{Sharma}, in the case
that $X$ is an approximately unital operator algebra.

3) \ 
One further result on $L$-structure: 
If  an operator
space $X$ is a right $L$-summand in its bidual, then any right
$L$-summand $Y$ of $X$ is a right $L$-summand in $Y^{**}$.  
Indeed if $X$ is the range of a left $L$-projection $P$ on $X^{**}$,
and if $Y$ is the range of a left $L$-projection $Q$ on $X$,
 then $Q^{**}$ and $P$ are in the left Cunningham algebra of $X^{**}$
 \cite[p.\ 8--9]{BZ}.
Note that $Q^{**} P = P Q^{**} P$ (since Ran$(Q^{**} P) \subset Y
\subset X$).  Since we are dealing with projections in a
$C^*$-algebra, we deduce that $P Q^{**} = Q^{**} P$. It follows that
$P(Y^{\perp \perp}) \subset Y$, and so $Y$ is a right $L$-subspace
of $X$ in the sense of \cite[Theorem 4.2]{Sharma}.  By that result,
$Y$ is a right $L$-summand in its bidual.    

\bigskip

{\bf Acknowledgements.}
We are grateful to the referee for  useful comments.

Note added November 2010: The questions posed in Section 3
have now been solved, with perhaps one exception, mostly as a consequence
of a deep recent theorem due to Read.  For more  details see  
``Operator algebras with contractive approximate identities"
by the second author and C. J. Read (arXiv:1011.3797).  We also remark 
that there is an obvious variant of Theorem 4.18 in terms of HSA's:
a separable operator algebra $A$ is $\sigma$-matricial iff $A$ is semiprime, a HSA in its bidual,
and every  HSA $D$ in $A$ with ${\rm dim}(D) > 1$, contains a nonzero
projection which is not an identity for $D$.

\end{document}